\documentclass[colorinlistoftodos,article,11pt,reqno]{article}
\usepackage{articlemacro}
\usepackage{tabularx}
\usepackage{etoc}

\newcommand{\RigSH}{\textup{RigSH}}
\newcommand{\RigH}{\textup{RigH}}
\newcommand{\RigDA}{\mathrm{RigDA}}
\newcommand{\DA}{\mathrm{DA}}
\newcommand{\RigDAA}{\RigDA_{\et}^{\AA^1}}
\newcommand{\RigDAAeff}{\RigDA_{\et}^{\AA^1,\mathrm{eff}}}
\newcommand{\RigDAB}{\RigDA_{\et}^{\BB^1}}
\newcommand{\RigDABeff}{\RigDA_{\et}^{\BB^1,\mathrm{eff}}}

\newcommand{\Sm}{\textup{Sm}}
\newcommand{\Rig}{\textup{Rig}}

\newcommand{\SmRig}{\textup{SmRig}}
\newcommand{\FSch}{\textup{FSch}}
\newcommand{\Pro}{\textup{Pro}}
\newcommand{\Cond}{\textup{Cond}}

\newcommand{\Spc}{\textup{Spc}}
\newcommand{\PSh}{\textup{PSh}}
\newcommand{\Sh}{\textup{Sh}}

\newcommand{\qcqs}{\textup{qcqs}}
\newcommand{\Nis}{\textup{Nis}}
\newcommand{\an}{\textup{an}}
\newcommand{\ft}{\textup{ft}}
\newcommand{\lft}{\textup{lft}}
\newcommand{\aup}{\mathrm{a}}

\newcommand{\Hup}{\mathrm{H}}
\newcommand{\rig}{\textup{rig}}
\newcommand{\eff}{\textup{eff}}
\newcommand{\lcomp}{\ell\textup{-comp}}

\newcommand{\K}{\textup{K}}

\newcommand{\Kcont}{\K^\textup{cont}}

\newcommand{\lbrr}{\lbrace\!\lbrace}
\newcommand{\rbrr}{\rbrace\!\rbrace}
\newcommand{\LB}{\L_{\BB^1}}
 
\newcommand{\Arig}{\AA^{1}}
\newcommand{\Anrig}{\AA^{n}}
\newcommand{\Brig}{\BB^{1}}
\newcommand{\Bnrig}{\BB^{n}}
\newcommand{\skp}[1]{\langle #1 \rangle}

\renewcommand{\to}[1][]{\overset{#1}{\rightarrow}}	
\newcommand{\To}[1][]{\overset{#1}{\longrightarrow}}	
\newcommand{\inj}[1][]{\overset{#1}{\hookrightarrow}}		

\renewcommand{\L}{\textup{L}}
\newcommand{\Lmot}{\L_{\textup{mot}}}
\newcommand{\es}{\emptyset}
\newcommand{\PrL}{\textup{Pr}^{\textup{L}}}
\newcommand{\PrLO}{\textup{Pr}^{\textup{L},\otimes}}
\newcommand{\CAlgPr}{\textup{CAlg}(\PrLO)}

\newcommand{\infcat}{category}
\newcommand{\infcats}{categories}

\newcommand{\carremap}[8]{\[\begin{xy}\xymatrix{#1\ar[r]^{#2}\ar[d]_{#3}&#4\ar[d]^{#5}\\#6\ar[r]_{#7}&#8}\end{xy}\]}
\newcommand{\carremaptag}[9]{\begin{align}\begin{xy}\tag{#1}\xymatrix{#2\ar[r]^{#3}\ar[d]_{#4}&#5\ar[d]^{#6}\\#7\ar[r]_{#8}&#9}\end{xy}\end{align}}

\date{}
\title{Towards $\mathbb{A}^1$-homotopy theory of rigid analytic spaces}
\author{Christian Dahlhausen and Can Yaylali}

\begin{document}
\maketitle
\begin{abstract}
\let\thefootnote\relax\footnote{2020 Mathematics Subject Classification: 14F42, 14G22}

To any rigid analytic space (in the sense of Fujiwara--Kato) we assign an $\mathbb{A}^1$-invariant rigid analytic homotopy category with coefficients in any presentable category. We show some functorial properties of this assignment as a functor on the category of rigid analytic spaces. Moreover, we show that there exists a full six-functor formalism for the precomposition with the analytification functor by evoking Ayoub's thesis. As an application, we prove mod $p$ rigidity for rigid spaces over $\mathbb{Q}_p$. Moreover, we prove the equivalence of $\mathbb{A}^1$-invariant sheaves and $\mathbb{B}^1$-invariant sheaves in the mod $\ell$ case.
\end{abstract}

\thispagestyle{empty}
\setcounter{tocdepth}{1}
\tableofcontents
\setcounter{tocdepth}{2}

\section{Introduction}
The $\AA^1$-homotopy category of schemes was constructed and studied by Morel\hyp{}Voevodsky \cite{MV1}.
In modern language, their homotopy category $\textup{H}(S)$ is the category of $\AA^1$\hyp{}invariant Nisnevich sheaves on smooth schemes over a base scheme $S$ with values in spaces. Under this construction they show, under some regularity assumptions, that connective algebraic K-theory can be represented as $\ZZ\times \BGL\in \textup{H}(S)$.
Moreover, after $\otimes$-inverting $\PP^1$, we obtain the stable homotopy category $\textup{SH}(S)$. In his thesis, Ayoub showed that the assignment $S\mapsto \textup{SH}(S)$ admits a full six-functor formalism \cite{AyoubThesis}. 
Ayoub deduces this from properties of the functor $\textup{SH}(-)$ such as $\AA^1$-invariance, $\PP^1$-stability, and the existence of the localisation sequence. In this article, we are going to extend these results to the analytic setting, using the rigid analytic line $\Arig$ as an interval.

\vspace{6pt}\noindent Motives on rigid analytic varieties have been constructed and studied by Ayoub\linebreak \cite{AyoubRig}. The idea is to follow Morel-Voevodsky's approach in constructing the stable homotopy category $\textup{RigSH}^{\Brig}(X)$ for a rigid analytic variety $X$. Ayoub defines the category of rigid motives by Nisnevich localisation, contracting the closed unit disc $\Brig$, and then $\otimes$-inverting the unit disc without the origin. In op.\!\! cit.\ it is shown that this definition yields a coefficient system in the sense that any morphism $f$ of rigid analytic varieties induces a colimit preserving functor $f^*$ on $\textup{RigSH}^{\Brig}(-)$ with the following properties.
\begin{enumerate}
    \item[(PF)] If $f$ is smooth, then $f$ admits a left-adjoint satisfying base change and projection formula.
    \item[(Loc)] Any closed immersion $i\colon Z\hookrightarrow X$ of rigid analytic varieties with open complement $j \colon U\hookrightarrow X$ induces a fibre sequence $j_\#j^* \to \id \to i_*i^*$ of endofunctors on $\textup{RigSH}^{\Brig}(X)$.
\end{enumerate}
 While the proof of (PF) is a formality, the proof of (Loc) needs more care and follows the idea of Voevodsky. These two properties and the definition immediately imply that on algebraic maps $X\rightarrow Y$ of rigid analytic varieties there exists a six-functor formalism \cite[Scholie 1.4.2]{AyoubThesis}. More recently, in joint work Ayoub--Gallauer--Vezzani generalise this result to arbitrary finite type maps of rigid analytic varieties \cite{AGV}.
\par
A drawback of this definition of $\textup{RigSH}^{\Brig}$ is that not all cohomology theories on rigid analytic varieties satisfy $\Brig$-homotopy invariance. For instance, continuous K-theory $\Kcont$ defined by Morrow \cite{Morrow}  and analytic K-theory defined by Kerz-Saito-Tamme \cite{kst-i} are not $\Brig$-invariant (Remark~\ref{Kan-not-B1-invariant--remark}). Hence, these theories cannot be represented in $\textup{RigSH}^{\Brig}$. On the other hand, letting $\Arig$ be the rigid affine line, analytic K-theory is $\Arig$-invariant \cite[Cor.~2.7]{kst-ii} and  -- assuming resolution of singularities -- continuous K-theory is $\Arig$-invariant for regular rigid analytic varieties \cite[Prop. 5.14]{kst-i}.
Doing motivic homotopy theory with $\Arig$ as the interval has already been studied 
by Sigloch \cite{sigloch-thesis}.
This suggests that one should define the category of motives on a rigid analytic variety $X$ as 
\[
    \RigSH(X)\coloneqq \Sh_{\Nis}^{\AA^1}(\SmRig_{X})[(\PP^{1}_{S},\infty)^{-1}],
\]
where $(\PP^{1}_{S},\infty)$ denotes the motive of $\PP^1_S$ pointed at $\infty$; this is equivalent to $\otimes$-inverting the Thom space of any vector bundle on $S$, see Section~\ref{sec.thom-motives-stabilisation}.

\vspace{6pt}\noindent 
In a subsequent paper with Yicheng Zhou, we establish a representability result for
the K-theory of rigid analytic spaces, namely analytic K-theory as defined and studied by Kerz--Saito--Tamme \cite{kst-i,kst-ii}. Since this K-theory takes values in pro-categories, it makes sense to consider motives with coefficients in an arbitrary symmetric monoidal stable presentable \infcat{} $\Vcal_{\textup{st}}$. Then we can apply the theory to the categories of pro-spectra and, since pro-categories do not behave that nicely, (light or $\kappa$-small) condensed spectra.

Let us fix a symmetric monoidal presentable category $\Vcal$ and its stabilisation $\Vcal_{\textup{st}}$. The motivating examples include the category $\Cond_{\kappa}(\Spc)$ of $\kappa$-small condensed spaces and the category $\Cond_{\kappa}(\Sp)$ of $\kappa$-small condensed spectra, for a fixed uncountable strong limit cardinal cardinal $\kappa$. The category $\RigSH(-,\Vcal)$ keeps the property (PF) above and perhaps more surprisingly also (Loc), even without stabilising. 

\begin{intro-theorem}[\protect{\ref{prop-smooth-BC-and-PF},~\ref{thm-gluing}}]
The assignment $X\mapsto \Sh^{\Arig}_{\Nis}(X,\Vcal),\ f\mapsto f^{*}$, from rigid analytic varieties to symmetric monoidal presentable $\Vcal$-linear categories satisfies (PF) and (Loc). 
\end{intro-theorem}

As an immediate consequence, we get a partial six-functor formalism on $\RigSH$ over a nonarchimedean field $K$.

\begin{intro-corollary}[\protect{\ref{thm-6ff}}]
   The assignment $X\mapsto \RigSH(X^{\an},\Vcal_{\textup{st}}),\ f\mapsto f^{*}$, from separated finite type $K$-schemes to symmetric monoidal stable presentable $\Vcal_{\textup{st}}$-linear categories admits a six-functor formalism satisfying base change, purity, and projection formula.
\end{intro-corollary}

An extension of the six-functor formalism to non-algebraic maps of rigid analytic varieties is not clear. In the classical rigid motivic theory of Ayoub, the rigid affine line and the closed unit disc are $\Brig$-homotopic \cite[Prop. 1.3.4]{AyoubRig}. This will not be the case, when working with $\Arig$-homotopy instead. This shows that we need a more careful treatment to obtain a full six-functor formalism for rigid spaces and cannot rely purely on the results of \cite{AGV}. Nevertheless, in future work we aim to establish a such full six-functor formalism.

\vspace{6pt}\noindent 
Let us pass to an application of our theory. Following the ideas of Bachmann and Ayoub--Gallauer--Vezzani, we show a rigidity statement. However, unlike \cite{AGV}, our theory enables us to work with $\ZZ/p$-coefficients. We will denote hypercomplete $\AA^1$-invariant \'etale motives with $\RigDA_{\et}^{\AA^1,\wedge}$ and the hypercomplete $\BB^1$-invariant \'etale motives with $\RigDA_{\et}^{\BB^1,\wedge}$. Let $K$ be a nonarchimedead field of residue characteristic $p\geq 0$.

\begin{intro-theorem}[Rigidity, \ref{Thm:Rigidity}]
\label{Thm:Rigidity}
    Let $\Lambda = \ZZ/N$ and assume that $\mathrm{char}(K) = 0$ or that $p\nmid N$. Then for any rigid space $X$ over $K$ the natural functor
	\[
		\Dcal(X_{\et},\Lambda) \To[\iota^*] \RigDA_{\et}^{\AA^1,\wedge}(X,\Lambda)
	\]
	is an equivalence.
\end{intro-theorem}

In particular, when $K=\QQ_p$ or $K=\CC_p$, for any $n\geq 1$ we have an equivalence
$$\Dcal(X_{\et},\ZZ/p^n) \To[\iota^*] \RigDA_{\et}^{\AA^1,\wedge,(,\eff)}(X,\ZZ/p^n),$$
which was previously not accessible. We obtain two immediate but important corollaries, a $p$-adic realisation and a comparison between $\BB^1$-motives and $\AA^1$-motives.

\begin{intro-corollary}[\ref{Thm:Real}]
     Let $\ell$ be any prime. Assume that $\mathrm{char}(K) = 0$ or that $p\neq\ell$.
      Then there exists a symmetric monoidal functor
	\[
		\hat{\rho}_\ell\colon \RigDA_{\et}^{\AA^1,\wedge}(X,\ZZ)\to \Dcal(X_{\et},\ZZ_{\ell})\coloneqq \lim_{n\geq 0}\Dcal(X_{\et},\ZZ/\ell^n).
	\]
\end{intro-corollary}

\begin{intro-corollary}[\ref{cor:comparison-with-B1-motives}]
    Assume that $p\nmid N$. Then the fully faithful embedding
    \[
		\RigDA^{\BB^1,\wedge(,\eff)}_{\et}(X,\Lambda) \subseteq \RigDA^{\AA^1,\wedge(,\eff)}_{\et}(X,\Lambda)
	\]
    is an equivalence.
\end{intro-corollary}

\vspace{6pt}\noindent\textbf{Acknowledgements.}
We heartily thank Joseph Ayoub for guiding us through his work on six-functor formalisms \cite{AyoubThesis} and rigid analytic motives \cite{AyoubRig} and Martin Gallauer for insightful explanations about his work \cite{AGV} and Brad Drew's work on coefficient systems \cite{DrewMHM}. 
We thank Marc Hoyois for helpful comments on Theorem~\ref{thm-gluing} and Alberto Vezzani for feedback on an earlier draft about rigidity.
Furthermore, we thank Ryomei Iwasa and Florian Strunk for helpful discussions around this paper's content.

Both authors were supported by the Deutsche Forschungsgemeinschaft (DFG)\linebreak through the Collaborative Research Centre TRR 326 \textit{Geometry and Arithmetic of Uniformized Structures}, project number 444845124. The second named author was furthermore supported through the Walter Benjamin Programme (DFG, project number 524431573) during his stay at the Laboratoire de Mathématiques d'Orsay. He wants to thank Vincent Pilloni for hosting him and all the other people giving a friendly and mathematically interesting environment.
Moreover, we thank Johannes Sprang for fostering this collaboration by hosting CD in Essen in January 2026.

\vspace{6pt}\noindent\textbf{Assumptions and notations.}
Throughout, we fix some inaccessible regular cardinal $\kgbar$. By \textit{small}, we will mean \textit{$\kgbar$-small}. 
A \emph{category} always means an $\infty$-category, i.e. an $(\infty,1)$-category in the sense of \cite{HTT}, and we identify 1-categories with their image under the nerve functor.
Let us fix some notation for the categories appearing in this article.
\begin{enumerate}
    \item[$\bullet$] $\Spc$, the category of small spaces, or equivalently small $\infty$-groupoids or small $\infty$-sets, or small anima (depending on your taste).
    \item[$\bullet$] $\Sp$, the category of spectra.
    \item[$\bullet$] $\PrL$, the category of presentable categories with colimit preserving functors.
    \item[$\bullet$] $\PrLO$, the category $\PrL$ endowed with the Lurie tensor product \cite[\S 4.8.1]{HA}.
\end{enumerate}
The term \emph{essentially unique} is short hand for the expression \emph{unique up to contractible choice}.

\section{Background in rigid analytic geometry}

We place ourselves in the same setting as Ayoub--Gallauer--Vezzani \cite[\S 1]{AGV} and use the notion of \emph{rigid spaces}\footnote{For the sake of brevity, we mostly use the terminology `rigid space' instead of the more accurate `rigid analytic space'.} as defined by Fujiwara-Kato \cite{fuji-kato} which goes back to Raynaud \cite{raynaud} and was also developed by Abbes \cite{egr}. 

\vspace{6pt}\noindent\textbf{Rigid spaces.} Let $\FSch$ be the category of adic formal schemes of finite ideal type in the sense of \cite[ch.~I, 1.1.14, 1.1.16]{fuji-kato}  and let $\FSch^\qcqs$ be its full subcategory of quasi-compact and quasi-separated formal schemes. The category $\Rig^\qcqs$ of quasi-compact and quasi-separated \emph{rigid spaces} is defined to be the 1\hyp{}categorical localisation of $\FSch^\qcqs$ by the set of admissible formal blow-ups. The  category $\Rig$ is  defined by gluing quasi-compact and quasi-separated rigid spaces along open immersions, see \cite[ch.~II, \S 2.2~(c)]{fuji-kato}. For a formal scheme $\Xcal$ denote by $\Xcal^\rig$ its associated rigid space.

\vspace{6pt}\noindent\textbf{Relation to adic spaces.} There exists a fully faithful functor $\textup{Adic}^\textup{unif}\inj\Rig$ from the category of uniform adic spaces into the category of rigid spaces sending $\Spa(R,R^+)$ to $\Spf(R^+)^\rig$ which is compatible with gluing along open immersions \cite[\S 1.2]{AGV}.
We call a rigid space \emph{adic} if it lies in the essential image of this functor.

\vspace{6pt}\noindent\textbf{Morphisms locally of finite type.}
A morphism $f\colon Y\to X$ of rigid spaces is said to be \emph{locally of finite type} if there exists an open cover $(X_i)_i$ of $X$ such that the restricted morphisms $f^{-1}(X_i)\to X$ admit formal models which are locally of finite type, i.e.\ locally given by algebras of topologically finite type, see  \cite[ch.~II, sec.~2.3]{fuji-kato}.
Note that an algebra of topologically finite type over a uniform Huber ring, is again a uniform Huber ring. It follows that, if $X$ is adic, then $Y$ is adic as well.

\vspace{6pt}\noindent\textbf{Smooth morphisms.}
A morphism $f\colon Y\to X$ of rigid spaces is called \emph{smooth} if, locally on $Y$ and $X$, there exists a formal formal model $\Ycal\to\Xcal$ which, locally on rings, is given as a composition $A\to A\langle t_1,\ldots,t_n\rangle \to[\phi] B$ where $\phi$ is rig-\'{e}tale, see \cite[1.3.3., 1.3.13]{AGV}.
It holds true that a morphism of adic rigid spaces is smooth in this sense if and only if it smooth in the sense of Huber, see \cite[1.3.14]{AGV} and the references there.

\vspace{6pt}\noindent\textbf{The Nisnevich topology.}
A family $(\Ycal_i\to\Xcal)_i$ in $\FSch$ is a \emph{Nisnevich cover} if the induced family $(\Ycal_{i,\sigma}\to \Xcal_\sigma)_i$ on special fibres is a Nisnevich cover of schemes. A family $(Y_i\to X)_i$ in $\Rig$ is a \emph{Nisnevich cover} if locally on $X$ the family admits a refinement that has a formal model which is a Nisnevich cover of formal schemes.

\vspace{6pt}\noindent\textbf{The rigid affine line.}
Let us quickly recall some background concerning the rigid affine line, for more details consider the lecture notes of Scholze--Weinstein \cite[Lecture~4]{weinstein-scholze} or Hübner \cite[Ex.~1.4.3, Ex.~1.10.4]{huebner-adic}. 
Let $(R,R^+)$ be a uniform Huber pair and set $S=\Spa(R,R^+)$.
Throughout this article, we work with the rigid affine line, seen as a rigid space\footnote{We will mostly work in the setting of rigid spaces and therefore omit a superscript indicating that $\Arig$ is not the schematic affine line.}
\[
    \Arig_{S} \coloneqq \Spa(\ZZ[T],\ZZ) \times S
\]
using the fully faithful embedding $\textup{Adic}^\textup{unif} \inj \Rig$ from above.\footnote{The ring $\ZZ[T]$ is equipped with the discrete topology. For the existence of fibre products we refer to \cite[ch.~II, sec.~2.4]{fuji-kato}.} This description of the affine line will make it easy to construct explicit homotopies. All of the computations (e.g.\ descriptions of homotopies) can be reduced to computations on $\Spa(\ZZ[T],\ZZ)$. 
\par
Analogously, we define the closed rigid unit disc as 
\[
    \Brig_{S} \coloneqq \Spa(\ZZ[T],\ZZ[T]) \times S = \Spa(R\langle T\rangle,R^+\langle T\rangle)
\]
where $R\langle T\rangle$ denotes the ring of those power series in $T$ whose coefficients tend to $0$. 

Now assume that $R$ is a Tate ring and let $\pi$ be a pseudo-uniformiser. Then we can describe the affine line as the union of closed unit discs with larger and larger radius
\begin{equation}
    \label{eq-Arig-colim}
    \Arig_{S} = \bigcup_{n\geq 0} \Spa(R\langle \pi^n T\rangle,R^+\langle \pi^n T\rangle).
\end{equation}

\section{Unstable  \texorpdfstring{$\Arig$}~-homotopy theory in rigid geometry}

\label{sec.unstable}
\vspace{6pt}
\noindent In this section, we want to analyse the category of $\Arig$-invariant Nisnevich sheaves on smooth rigid spaces, also called the unstable rigid homotopy category. The upshot is that this category will define a pullback formalism in the sense that, for any smooth morphism $f$ of rigid spaces, the pullback $f^*$ admits a left-adjoint satisfying base change and projection formula. But we will go a bit further and prove that we even have a localisation sequence in this more general setting.

\vspace{6pt}\noindent\textbf{Notation.} Throughout this section, let $(R,R^+)$ be a uniform Huber pair with associated rigid space $S=\Spa(R,R^+)$ and let $\Dcal$ be a presentable category of coefficients.

\subsection{Homotopy invariant presheaves on rigid spaces}

As mentioned earlier, there are several reasonable choices of an interval object to do homotopy theory with rigid spaces such as the closed unit disc $\Brig$ and the rigid affine line $\Arig$. 

\begin{defi} \label{Aan-B1-invariance--defin}
We say that a presheaf $F\in\Fun(\Rig_S\op,\Dcal)$ is \ldots
\begin{enumerate}
	\item $\Brig$\emph{-invariant} if for every $X\in\Rig_S$ the canonical map $F(X) \to F(X\times\Brig)$ is an equivalence in $\Dcal$.
	\item $\Arig$\emph{-invariant} if for every $X\in\Rig_S$ the canonical map $F(X)\to F(X\times\Arig)$ is an equivalence in $\Dcal$.
\end{enumerate}
\end{defi}

\begin{defi}
Let $\Ccal$ be a site. A presheaf $F\in\Fun(\Ccal\op,\Dcal)$ is said to be a \emph{sheaf} if for every object $X$ of $\Ccal$ and every covering sieve $U\inj X$ the canonical map $F(X)\to F(U)$ is an equivalence in $\Dcal$; here $F(U)$ is defined via the equivalence
	\[
	\Fun^\mathrm{lim}(\PSh(\Ccal)\op,\Dcal) \To \Fun(\Ccal\op,\Dcal).
	\]
Thus if we write $U \simeq \colim_iU_i$ as a colimit of representables, then $F(U) \simeq \lim_iF(U_i)$.
\end{defi}

\begin{lemma} \label{B1-implies-Aan-invariance--lem}
Every $\Brig$-invariant sheaf on $\Rig_S$ is $\Arig$-invariant.
\end{lemma}
\begin{proof}
We can identify $\Arig$ with the colimit $\colim_{t\mapsto\pi t}\Brig$ in the category $\Rig_S$ where $t$ is a parameter for $\Brig=\Spa(k\skp{t}),k^\circ\skp{t})$. Since $F$ is a sheaf (and hence preserves this colimit) and since a colimit of open immersions commutes with fibre products, we have for every sheaf $F$ equivalences
	\begin{align*}
	F(X\times\Arig) \simeq F(X\times\colim_{t\mapsto\pi t}\Brig)
	\simeq F(\colim_{t\mapsto\pi t}(X\times\Brig))
	\simeq \lim_{t\mapsto\pi t} F(X\times\Brig).
	\end{align*}
Hence the map $F(X) \to F(X\times\Arig)$ is an equivalence if the map $F(X) \to F(X\times\Brig)$ is an equivalence.
\end{proof}

\begin{defi}
Denote by $\PSh^{\Arig}(\Rig_k,\Dcal)$ the full subcategory of the category of\linebreak presheaves $\PSh(\Rig_k,\Dcal)$ spanned by $\Arig$-invariant presheaves.
\end{defi}

\begin{remark} \label{Kan-not-B1-invariant--remark}
There are $\Arig$-invariant sheaves which are not $\Brig$-invariant. For instance, analytic K-theory $\Kcont \colon \Rig_k\op \to \Pro(\Sp^+)$ is a sheaf \cite[6.15]{kst-i} \cite[4.3]{kst-ii} which is $\Arig$-invariant \cite[4.4]{kst-ii}. However, assuming resolutions of singularities, it is not $\Brig$-invariant as it agrees with continuous K-theory \cite[6.19]{kst-i} which on $\pi_0$ is equivalent to $\K_0$ \cite[5.10]{kst-i} which has the non-$\Brig$-invariant Picard group as a direct summand \cite[Example~2]{kst-bass-quillen}.
\end{remark}

\begin{lemma} \label{homotopy-invariantification--lem}
For every cocomplete category $\Dcal$ there exists a localisation functor
	\[
	\L_{\Arig} \colon \PSh(\Rig_S,\Dcal) \To \PSh^{\Arig}(\Rig_S,\Dcal)
	\]
which is left-adjoint to the inclusion. It has the following explicit description: for every $F\in\PSh(\Rig_k,\Ccal)$ and every $X\in\Rig_S$ there is a functorial equivalence
	\[
	(\L_{\Arig} F)(X) \simeq \colim_{\Delta\op} F(X\times \Delta^{\mathrm{an},\bullet})
	\]
where $\Delta^{\mathrm{an},\bullet}$ is the analytification of its algebraic analogue.
\end{lemma}
\begin{proof}
For the proof, denote by $\textup{H} F$ the presheaf whose values are the right-hand side of the desired identification. This yields a natural transformation $\alpha \colon \id \to \textup{H}$  of presheaves with values in $\Dcal$. Assume that the following conditions hold:
\begin{enumerate}
	\item[(i)] For every presheaf $F$, the presheaf $\textup{H} F$ is $\Arig$-invariant.
	\item[(ii)] For every $\Arig$-invariant presheaf $F$, the map $\alpha_F \colon F \to \textup{H} F$ is an equivalence.
\end{enumerate} 
These conditions imply for every presheaf $F$ that both maps $\textup{H}(\alpha_F),\alpha_{\textup{H} F} \colon \textup{H} F \to \textup{H}\textup{H} F$ are equivalences. Then the desired claim follows from \cite[5.2.7.4]{HTT}. Hence it suffices to check conditions (i) and (ii).

(i) Let $X\in\Rig_k$ and let $p\colon X\times\Arig \to X$ the projection and $\sigma \colon X \to X\times\Arig$ the zero section. As $p\circ\sigma=\id$, it remains to show that the induced map $p^*\circ\sigma^*$ on $(\textup{H} F)(X\times\Arig)$ is equivalent to the identity. For this purpose, we claim that the maps
	\[
	p^*\circ\sigma^*,\id \colon X\times\Arig\times\Delta^{\mathrm{an},\bullet} \rightrightarrows X\times\Arig\times\Delta^{\mathrm{an},\bullet}
	\]
are simplicial homotopic (see Remark~\ref{simplicial-homotopic--remark} below). In degree $n$, a homotopy $h_n$ on $X\times\Arig\times\Delta^{\mathrm{an},n}$ is given by $h_n(f)(t) = (\sum_{j\in f^{-1}(1)}t_j)\cdot t$ for a map $f \colon [n] \to{} [1]$. We have $i_0^*h(t)=0$ and $i_1^*h(t)=t$.

(ii) We follow the classical reasoning \cite[ch.~IV, Lemma~11.5.1]{k-book}. Let $F$ be an $\Arig$-invariant presheaf. Then for every $X\in\Rig$ there is a morphism $X\times\Delta^{\an,\bullet} \to X$ of simplicial rigid spaces which is induced by the projections. For every $n\in\Delta\op$ the induced map $F(X) \to F(X\times\Delta^{\an,n})$ is an equivalence. Hence the map
\[ F(X) = \colim_{\Delta\op} F(X) \To \colim_{\Delta\op} F(X\times\Delta^{\an,n}) = \textup{H}F(X) \]
is an equivalence.

\end{proof}

\begin{remark} \label{simplicial-homotopic--remark}
Let $\Ccal$ be a \infcat{} and $\mathrm{s}\Ccal$ the simplicial objects in $\Ccal$. We have a diagram
	\[
	\Delta \overset{i_0}{\underset{i_1}{\rightrightarrows}} \Delta_{/[1]} \to[\delta] \Delta
	\]
where $i_k([n]) = ([n]\to{}[0]\to[k][1])$ and $\delta([n]\to{}[1]) = [n]$. A \emph{simplicial homotopy} between two maps $f,g \colon X_\bullet \to Y_\bullet$ in $\mathrm{s}\Ccal$ is a map $h \colon \delta^*(X_\bullet) \to \delta^*(Y_\bullet)$ in $\Fun((\Delta_{/[1]})\op,\Ccal)$ such that $i_0^*(h)=f$ and $i_1^*(h)=g$. In this case, if $\Ccal$ is cocomplete, the induces maps $|f|,|g| \colon |X_\bullet| \to |Y_\bullet|$ on geometric realisations are homotopic in $\Ccal$. This follows since the map $\colim_{(\Delta_{/[1]})\op} \delta^*(X_\bullet) \to |X_\bullet| = \colim_{\Delta\op} X_\bullet$ is an equivalence. 
\end{remark}

\subsection{The unstable analytic motivic category}

\begin{defi} \label{RigH--definition}
Let $S$ be a rigid space. The \emph{unstable rigid motivic homotopy category with coefficients in $\Dcal$} is the reflective subcategory 
\[ \RigH(S,\Dcal) := \Sh^{\Arig}_\Nis(\SmRig_S,\Dcal) \]
of $\Sh_\Nis(\SmRig_S,\Dcal)$ which is spanned by $\Arig$-invariant Nisnevich sheaves. We have the following diagram of localisation functors:

$$ 
\begin{tikzcd}
    & \Sh_\Nis(\SmRig_S,\Dcal) \arrow[dr,"\L_{\Arig}^\#"]& \\
     \PSh(\SmRig_S,\Dcal) \arrow[ur,"\L_\Nis"] \arrow[dr,"\L_{\Arig}"] \arrow[rr,"\Lmot"] && \RigH(S,\Dcal) \\
     & \PSh^{\Arig}(\SmRig_S,\Dcal) \arrow[ur,"\L_\Nis^\#"]
\end{tikzcd} 
$$

We note that the \textit{motivic localisation} $\Lmot$ is given by the colimit of functors 
\[ \Lmot \simeq \colim( \L_\Nis\L_{\Arig} \to \L_\Nis\L_{\Arig}\L_\Nis\L_{\Arig} \to \ldots ). \]
\end{defi}

\subsection{Analytification for motivic spaces}
\label{sec:Analytification}

Let $K$ be a nonarchimedean field and let $S$ be a $K$-scheme of finite type. In this section, we construct an analytification functor from the algebraic stable homotopy category $\textup{H}(S,\Dcal)$ to the rigid analytic $\Arig$-homotopy category $\RigH(S^\an,\Ecal)$ depending on a \emph{change of coefficients functor} $\Dcal\to\Ecal$ between symmetric monoidal presentable \infcats. Later on, we apply this with the functors $\Spc\inj\Cond(\Spc)$ and $\Sp\inj\Cond(\Sp)$.

\vspace{6pt}\noindent\textbf{Notation.}
In this section, let $R_0$ be an adic ring with ideal of definition $I$. We set $S:=\Spec(R_0)\setminus\textup{V}(I)$ and $S^\an:=\Spf(R_0)^\rig$.
In the special case where $R_0$ is a ring of definition of a Huber pair $(R,R^+)$, we  have $S=\Spec(R)$ and $S^\an=\Spa(R)=\Spa(R,R^+)$. For any $R$-algebra $A$ we set $A^+$ as the integral closure of (the image of) $R^+$ in $A$ and write $\Spa(A):=\Spa(A,A^+)$ for simplicity.

\begin{lemma} \label{analytification-functor--lem}
There exists an analytification functor $(-)^\an \colon \Sch_S^\lft \to \Rig_{S^\an}^\lft$ sending the algebraic affine line $\Arig_S$ to the rigid analytic affine line $\Arig_{S^\an}$. In particular, for every $X\in\Sch^\lft_S$ we get an analytification functor $(-)^\an \colon \Sch_X^\lft \to \Rig_{X^\an}^\lft$. By restriction, we get a functor $(-)^\an\colon\Sm_X\to\SmRig_{X^\an}$.
\end{lemma}
\begin{proof}
For general rigid spaces the analytification functor has been constructed by\linebreak Fujiwara--Kato \cite[ch.~2, sec.~9.1.]{fuji-kato}. We give a sketch of the construction in the adic setting; for a full proof we refer to Hübner's lecture notes \cite[\S 1.10]{huebner-adic}.  

So, assuming that $S=\Spec(R)$ for a Huber pair $(R,R^+)$, the analytification functor is locally given by the assignment
\[
	R[t_1,\ldots,t_n]/(f_1,\ldots,f_k) \mapsto \bigcup_{r\geq 1} \Spa\bigl( R\skp{t_1,\ldots,t_n}_r/(f_1,\ldots,f_k)\bigr)
\]
where $R\skp{t_1,\ldots,t_n}_r$ denotes the subring of $R\llbracket t_1,\ldots,t_n\rrbracket$ consisting of power series\linebreak $\sum_{i\in\NN^n}a_it^i$ such that $a_ir^i$ converges to zero as $|i|\to\infty$. This assignment is compatible with gluing and extends to the desired functor.

Alternatively, one can associate with any $X\in\Sch_S^\ft$ a natural presheaf $h_\an(X)$ on $\Rig_K$ which turns out to be representable by a rigid space $X^\an$ yielding the analytification functor, see  \cite[\S 1.1.3]{AyoubRig}.

\end{proof}

\begin{lemma} \label{analytification-presheaves--lem}
The analytification functor $(-)^\an \colon \Sch_S^\lft \To \Rig_{S^\an}^\lft$ extends to an essentially unique colimit-preserving functor
\[
\PSh(\Sch_S^{\lft}) \To \PSh(\Rig_{S^\an}^{\lft})
\]
compatible with the respective Yoneda embeddings. Furthermore, this functor is monoidal with respect to the cartesian monoidal structure. The analogous statement for the functor $(-)^\an\colon\Sm_S\to\SmRig_{S^\an}$ holds true, too.
\end{lemma}
\begin{proof}
Restricting along the Yoneda embedding $\Sch_S^\ft \inj \Fun(\Sch_S^{\ft,\opp},\Spc)$ induces for any \infcat{} $\Ccal$ which admits small colimits an equivalence \cite[Prop. 5.1.5.6]{HTT}
\[
\Fun^\mathrm{L}\bigl( \PSh(\Sch_S^{\ft}), \Ccal \bigr) \To[\simeq] \Fun(\Sch_S^\ft,\Ccal)
\]
where $\Fun^\mathrm{L}$ denotes the full subcategory of all functors that preserve small colimits. Hence the composition $\Sch_S^\lft \To[(-)\an] \Rig_{S^\an}^\lft \inj \PSh(\Rig_{S^\an}^\lft)$ extends to an essentially unique functor 
\[
    \PSh(\Sch_S^{\lft}) \To \PSh(\Rig_{S^\an}^{\lft})
\]
via left Kan extension. Note that the analytification functor preserves finite products and hence also the left Kan extension of this functor along the Yoneda embedding.
\end{proof}

\begin{lemma} \label{analytification-homotopy-invariant--lem}
The functor from Lemma~\ref{analytification-presheaves--lem} fits into a commutative square
\carremap{\Fun(\Sch_S^{\lft,\opp},\Spc)}{}{\L_{\Arig}}{\Fun(\Rig_{S^\an}^{\lft,\opp},\Spc)}{\L_{\Arig}}{\Fun^{\Arig}(\Sch_S^{\lft,\opp},\Spc)}{}{\Fun^{\Arig}(\Rig_{S^\an}^{\lft,\opp},\Spc)}
The analogous statement for the functor $(-)^\an\colon\Sm_S\to\SmRig_{S^\an}$ holds true, too.
\end{lemma}
\begin{proof}
The functor $\Fun(\Sch_S^{\lft,\opp},\Spc) \To \Fun(\Rig_{S^{\an}}^{\lft,\opp},\Spc)$ sends a morphism $\mathrm{pr}_1\colon X\times\Arig_S\to X$ to the morphism $\mathrm{pr}_1\colon X^\an\times\Arig_{S^\an}\to X^\an$ which becomes an equivalence after applying $\L_{\Arig}$, hence we get the desired functor from the universal property of localisation. 
\end{proof}

\begin{lemma} \label{analytification-sheaves--lem}
The functor from Lemma~\ref{analytification-presheaves--lem} fits into a commutative square
\carremap{\Fun(\Sch_S^{\lft,\opp},\Spc)}{}{\L_\Nis}{\Fun(\Rig_{S^\an}^{\lft,\opp},\Spc)}{\L_\Nis}{\Fun^{\Nis}(\Sch_S^{\lft,\opp},\Spc)}{}{\Fun^{\Nis}(\Rig_{S^\an}^{\lft,\opp},\Spc)}
The analogous statement for the functor $(-)^\an\colon\Sm_S\to\SmRig_{S^\an}$ holds true, too.
\end{lemma}
\begin{proof}
The functor $\Fun(\Sch_K^{\lft,\opp},\Spc) \To \Fun(\Rig_{S^\an}^{\lft,\opp},\Spc)$ sends Nisnevich covers to Nisnevich covers so that we get the desired functor from the universal property of localisation. 
\end{proof}

\begin{lemma} \label{analytification-motivic-spaces--lem}
There exists an analytification functor $\textup{H}(S) \to \RigH(S^\an)$ fitting into a commutative square 
\carremap{\Sm_S}{(-)^\an}{}{\SmRig_{S^\an}}{}{\textup{H}(S)}{}{\RigH(S^\an).}
\end{lemma}
\begin{proof}
    This follows from Lemmas \ref{analytification-homotopy-invariant--lem} and \ref{analytification-sheaves--lem}.
\end{proof}

\begin{rem}[Change of coefficients] 
\label{analytification-coefficient-change--rem}
Let $\lambda\colon\Dcal\rightleftarrows\Ecal\colon\rho$ be an adjunction. For any category $\Ccal$, since we can compute equivalences in functor categories objectwise, we get an induced adjunction
\[ \lambda_* \,\colon\, \Fun(\Ccal\op,\Dcal) \rightleftarrows\Fun(\Ccal\op,\Ecal) \,\colon\, \rho_*.\]
If $\Ccal$ carries a topology, we get an induced adjunction
\[ \tilde{\lambda}_* \,\colon\, \Sh(\Ccal,\Dcal) \rightleftarrows\Sh(\Ccal,\Ecal) \,\colon\, \tilde{\rho}_*\]
with $\tilde{\lambda}_* \simeq \L_\Ecal\circ\lambda_*$ and $\tilde{\rho}_*\simeq\rho_*$ (omitting the inclusion functors from presheaves to sheaves), since we can patch adjunctions together. 

If $\Ccal\op\in\{\Sch_S^{\lft},\Sm_S,\Rig_{S^\an}^{\lft},\SmRig_{S^\an}\}$, then we get an induced adjunction
\[ \hat{\lambda}_* \,\colon\, \Sh^{\Arig}_{(\Nis)}(\Ccal,\Dcal) \rightleftarrows\Sh^{\Arig}_{(\Nis)}(\Ccal,\Ecal) \,\colon\, \hat{\rho}_*\]
on the full subcategories of $\Arig$-invariant (Nisnevich) sheaves.
Combining with Lemma~\ref{analytification-motivic-spaces--lem} we get adjunctions
\[
\textup{H}(S,\Spc) \rightleftarrows \RigH(S^\an,\Spc) \rightleftarrows \RigH(S^\an,\Pro^\omega(\Spc)).
\]
In the case where $\Dcal$ and $\Ecal$ are commutative algebra objects in $\PrL$ and $\lambda$ is symmetric monoidal, we can give an alternative description of the adjunction $\hat{\lambda}_*\rightleftarrows\hat{\rho}_*$, so that we obtain a monoidal adjunction.
We will make this precise in Section~\ref{sec.change.of.coeff.monoidal}. The main application of this alternative description will be that $\hat{\lambda}$ will be a symmetric monoidal functor.
\end{rem}

\subsection{Functoriality}

We have defined the motivic homotopy $\RigH(-,\Vcal)$ category with coefficients in any presentable category $\Vcal$ (Definition~\ref{RigH--definition}). However, it suffices to prove most results on functoriality for coefficients in $\Spc$.
This eases computations as we can easily reduce ourselves to the case of presheaves and localise accordingly. But it is not a drawback as we will see later in Lemma~\ref{lem-effective-coeff} that this implies that also $\RigH(-,\Vcal)$ will satisfy the same functorial properties as $\RigH$. This will include the localisation sequence proven in Theorem~\ref{thm-gluing}.

Let $f\colon S\to T$ be a morphism of rigid spaces. Then there is a base change\footnote{Here the fibre product denotes the fibre product of rigid space. On affinoids, this corresponds to the completed tensor product. We drop the completion from our notation.} functor 
\[ b \colon \SmRig_T\to\SmRig_S, X\mapsto X\times_TS\eqqcolon X_{S},\]
inducing by precomposition a functor $f_* := b^* \colon \PSh(\SmRig_S) \to \PSh(\SmRig_T)$ which restricts to a functor 
\[ f_* \colon \RigH(S) \To \RigH(T). \]
This functor preserves limits and is accessible\footnote{On presheaves, this is clear. Then the left-adjoint of $f_{*}$ on $\RigH$ is given by $\L_{\Arig}^{\#}\L_{\Nis}f^{*}$. By abuse of notation, we will again denote this left-adjoint with $f^{*}$.}, hence admits a left-adjoint functor
\[ f^* \colon \RigH(T) \To \RigH(S). \]
The functor $f^*$ is a left Kan extension and, for $F\in\RigH(T)$ and $X\in\SmRig_{S}$, one has an equivalence \cite[Prop. 4.3.3.7]{HTT}
\[ (f^*F)(X) \simeq \colim_{\SmRig_T\ni Y, Y_{S}\to X}F(Y). \]
In particular, $(f^*F)(X\to S) \simeq F(X\to S\to[f] T)$ whenever $f$ is smooth. In this case the functor $f^{*}$ is induced by the restriction along the postcomposition functor $\SmRig_{T}\to \SmRig_{S}$ and, as above, we see that $f$ admits a left-adjoint
\[ f_{\sharp}\colon \RigH(S)\to \RigH(T) \]
by formal reasons.

\begin{lemma}
\label{lem.Nis-localisation}
The category $\Sh_\Nis(\SmRig_S)$ is a topological localisation of $\PSh(\SmRig_S)$. In particular, the category $\Sh_\Nis(\SmRig_S)$ is an $\infty$-topos and the localisation functor\linebreak $\L_\Nis \colon \PSh(\SmRig_S) \to \Sh_\Nis(\SmRig_S)$ is accessible and left-exact.
\end{lemma}
\begin{proof}
Consider \cite[6.2.1.4]{HTT} for the definition of a topological localisation. Such localisations are always accessible and left-exact \cite[6.2.1.6]{HTT}. In general, the category of sheaves for Grothendieck topologies are topological localisations of the category of presheaves \cite[6.2.2.7]{HTT}.
\end{proof}

\begin{lem}
\label{lem.loc.cartesian}
	The motivic localisation $\Lmot$ preserves finite products. 
\end{lem}
\begin{proof}
	For $\L_{\Nis}$ this follows from Lemma~\ref{lem.Nis-localisation} and for $\L_{\Arig}$ this follows from \cite[Prop. 3.4]{Hoy-EQ}. 
    
    For $\Lmot$ let us set $\L^0\coloneqq \id$ and $\L^n\coloneqq \L_{\Nis}\L_{\Arig}\L^{n-1}$ for $n\geq 1$. In general $\Lmot$ is equivalent to $\colim_{n\geq 0} \L^n$, here the colimit is in the $\infty$-category of endofunctors of $\PSh(\SmRig_S)$.
    Since colimits in functor categories are taken pointwise, we see that $\Lmot$ commutes with finite products if and only if each $\L^n$ commutes with finite products, and filtered colimits in $\infty$-topoi commute with finite products. Both assertions are clear by assumption and \cite[Ex. 7.3.4.7]{HTT}. In fact, the latter statement shows that filtered colimits in $\infty$-topoi are left exact and thus also commute with the formation of pullbacks.
\end{proof}

\begin{rem}[Monoidal structure on $\RigH$]
\label{rem-RigH-monoidal}
	Let $S$ be a rigid space. The category of spaces $\Spc$ together with its cartesian monoidal structure induces a symmetric monoidal structure on $\PSh(\SmRig_S)$ by the pointwise tensor product \cite[Rem. 2.1.3.4]{HA}. By Lemma~\ref{lem.loc.cartesian}, the localisation $\Lmot$ is compatible with this monoidal structure, inducing a  symmetric monoidal structure on $\RigH(S)$ \cite[Prop. 2.2.1.9]{HA}. Note that the tensor product on the presheaf level commutes with arbitrary colimits and thus also on $\RigH(S)$. 
	If $f\colon S\to T$ is a map of rigid spaces, then $f^{*}$ is a symmetric monoidal functor.\footnote{This is clear as $f^{*}$ on presheaves certainly preserves finite products and as $\Lmot$ commutes with finite products, we can reduce to this case. But this is clear, as $f^{*}$ can be written as a filtered colimit (the fact that this colimit is filtered follows from \cite[00X3]{stacks-project}).}

    Let us add that $\RigH(S)$ is obtained via localisation of an $\infty$-topos at a small set of morphisms. In particular, $\RigH(S)$ is presentable and even more it is a commutative algebra object in $\PrL$.
\end{rem}

\begin{prop} \label{prop-smooth-BC-and-PF}
	Let $f\colon S\rightarrow T$ be a smooth morphism of rigid spaces.
	\begin{enumerate}
		\item The pullback $f^{*}\colon \RigH(T)\rightarrow \RigH(S)$ admits a left-adjoint $f_{\sharp}$.
		\item (Smooth projection formula) The canonical map 
			\[ f_{\sharp}(f^{*}M\otimes N)\to M\otimes f_{\sharp}N \]
			is an equivalence for all $M\in\RigH(T)$ and $N\in\RigH(S)$.
		\item (Smooth base change) Let $g\colon X\rightarrow T$ be a morphism of rigid spaces and let 
		\[ \begin{tikzcd}
			X_{S}\arrow[r,"f'"]\arrow[d,"g'"]& X\arrow[d,"g"]\\
			S\arrow[r,"f"]&T
		\end{tikzcd} \]
		be a pullback diagram. Then the exchange map $f'_{\sharp}g'^{*}\to g^{*}f_{\sharp}$ is an equivalence.
	\end{enumerate}
\end{prop}
\begin{proof}
	The first assertion follows from the discussion in the beginning. For (2) and (3) we may use that $\L_{\Arig}\L_{\Nis}$ commutes with finite products and thus, we can check them on the level of presheaves, where they follow from the construction.
\end{proof}

\begin{rem}
    Let us note that Proposition~\ref{prop-smooth-BC-and-PF} (3) is equivalent to the right-adjointability of the transposed diagram, i.e. $f'_{\sharp}g'^{*}\to g'^{*}f'_{\sharp}$ is an equivalence if and only if the exchange map $f^{*}g_{*}\to g'_{*} f^{'*}$ is an equivalence.
\end{rem}

\begin{prop} \label{prop-nisnevich-descent-RigH}
The following assignments define sheaves for the Nisnevich topology on $\Rig$ with values in $\Pr^{L}$.
\begin{enumerate}
	\item The assignment $S\mapsto\Sh_\Nis(\SmRig_S), f\mapsto f^*$.
	\item The assignment $S\mapsto\RigH(S), f\mapsto f^*$.
\end{enumerate}
\end{prop}
\begin{proof}
The first part follows from \cite[Prop. 2.3.7]{AGV}.\par 
For the second part, let us note that it is enough to show that, for any Nisnevich cover $f\colon U\to X$ and any Nisnevich sheaf $F$ such that $f^*F$ is $\Arig$-invariant, also $F$ is $\Arig$-invariant. This follows since 
\[ \Homline(\Arig_{U},f^*F) \simeq f^*\Homline(\Arig_X,F) \]
and since $f^*$ is conservative, where $\Homline$ denotes the internal $\Hom$ in $\Sh_\Nis(\SmRig_S)$ (this is completely analogous to the proof of \cite[Thm. 2.3.4]{AGV}).
\end{proof}

In the rest of this subsection, we want to prove that for a closed immersion $i\colon Z\inj S$ the pushforward $i_{*}$ preserves weakly contractible colimits. This will follow via direct computation.\par
Let $S$ be a rigid space. We denote by $\es/S$ the initial object in $\SmRig_{S}$. The full subcategory spanned by those presheaves $F\in \PSh(\SmRig_{S})$ such that $F(\es/S)\simeq \ast$ is denoted by $\PSh_{\es}(\SmRig_{S})$. Equivalently, we can localise $\PSh(\SmRig_{S})$ at the morphism $\es \to \es/S$, where $\es$ is the empty functor.  In particular, the inclusion $\PSh(\SmRig_{S})\to \PSh_{\es}(\SmRig_{S})$ admits a left-adjoint, denoted by $\L_{\es}$.

\begin{rem}
\label{rem.Pes.gen}
	A presheaf $F\in\PSh(\SmRig_{S})$ is contained in $\PSh_{\es}(\SmRig_{S})$ if and only if its essential image is weakly contractible. Hence the category $\PSh_{\es}(\SmRig_{S})$ is generated by representables under weakly contractible colimits (i.e.\ colimits of shape $I$ where the map $I\to *$ is cofinal).
\end{rem} 

\begin{lem}
\label{lem.i.compatible}
	Let $i\colon Z\hookrightarrow S$ be a closed immersion of rigid spaces. The functor\linebreak $i_{\es *}\colon \PSh_{\es}(\SmRig_{Z})\rightarrow  \PSh_{\es}(\SmRig_{S})$ commutes with $\L_{\Nis}$ and $\Lmot$.
\end{lem}
\begin{proof}
That $i_{\es*}$ commutes with Nisnevich-sheafification follows with the same arguments as  \cite[Lem. 1.4.19]{AyoubRig}. For the sake of completion let us recall the proof.\par 
	It is enough to show $\L_{\Nis}\circ i_{\es*}\simeq i_{\es *}\circ \L_{\Nis}$. For this let $F\in\PSh_{\es}(\SmRig_{Z})$. Let $U$ be a rigid space over $S$. We can define
	\[
		(i_{\es*}F)^{+}(U)\coloneqq \colim_{(U_{i})_{i}\in\textup{Cov}_{\Nis}(U)}\lim_{\Delta}i_{\es*}F(\Cv(\coprod_{i} U_{i}/U)_{\bullet}),
    \]
	where the colimit runs over all Nisnevich coverings $(U_{i}\to U)_{i}$. The sheafification is then given by applying $(-)^{+}$ infinitely many times (via transfinite induction - cf. proof of \cite[Prop. 6.2.2.7]{HTT}). In the proof of \cite[Lem. 2.2.4]{AGV} it is shown that if $U\times_{S}Z$ is non-empty, then any Nisnevich covering of $U\times_{S}Z$ can be refined by the base change of a Nisnevich covering of $U$. In particular, the base change map $\textup{Cov}_{\Nis}(U)\to \textup{Cov}_{\Nis}(U\times_{S}Z)$ is cofinal. Thus, the definition of $\L_{\Nis}$ via the $(-)^{+}$-construction shows
	\[
		\L_{\Nis}(i_{\es*}F)(U)\simeq (\L_{\Nis}F)(U\times_{S}Z)\simeq i_{\es*}(\L_{\Nis}F)(U).
	\]
	\par 
	For the commutativity with $\L^{\#}_{\Arig}\circ \L_{\Nis}$ it is enough to show that for all $U\in\SmRig_{Z}$, we have that $p\colon i_{\es*}\Arig_{U}\to i_{\es*}U$ is an $\Arig_{S}$-equivalence (similar to the proof of \cite[Prop. 1.4.18]{AyoubRig}). The $0$-section of $\Arig_{U}\to U$ induces a section $s$ of $p$. We claim that also $s\circ p$ is $\Arig$-homotopy equivalent the identity on $i_{\es *}\Arig_{U}$. The composition
	\[
		\Arig_{S}\otimes i_{\es *}\Arig_{U}\to i_{\es *}\Arig_{Z}\otimes  i_{\es *}\Arig_{U}\simeq i_{\es *}(\Arig_{Z}\times_{Z}\Arig_{Z}\times_{Z} U)\xrightarrow{m}i_{\es *}(\Arig_{Z}\times_{Z} U), 
	\]
	where $m$ is the multiplication, naturally yields such a homotopy. In particular, $p$ is an $\Arig$-homotopy equivalence.
\end{proof}

\begin{rem}
\label{rem.closed.no-colim}
	Let us note that $i_{\es *}$ cannot preserve the initial object, when $U\coloneqq S\setminus Z\neq \es$. This is because $i_{\es*}(\es/Z)(U)\simeq \es/Z(\es/Z)\simeq \ast \not\simeq \es = \es/S(U)$. But $i_{\es *}$ preserves weakly contractible colimits, as $\PSh_{\es}(\SmRig_{Z})$ is freely generated under weakly contractible colimits in $Z$ (and similarly for $S$).
\end{rem}

The remark above immediately implies the following.

\begin{cor} \label{cor-i_*-preserves-colimits}
Let $i\colon Z\inj S$ be a closed immersion of rigid spaces. Then the functor $i_* \colon \RigH(Z) \to \RigH(S)$ preserves weakly contractible colimits.
\end{cor}

The only obstruction for $i$ above not preserving colimits is that in $\RigH(S)$ the initial and the final object may not agree. This is certainly a "problem" inside $\Spc$, which will go away after stabilising. 

\subsection{Gluing}
In this section, we want to prove the existence of a
gluing theorem in the unstable setting which yields a localisation sequence in the stable setting. Such a theorem has been first proven by Morel--Voevodsky in the algebraic case \cite[Thm.~2.21]{MV1}.  In $\Brig$-homotopy theory this result is proven by Ayoub \cite{AyoubRig}. We will follow the proof structure presented by Hoyois \cite[\S 4]{Hoy-EQ}.

\vspace{6pt}\noindent\textbf{Notation.} In this subsection, we fix a rigid space $S$. 

\begin{defi} \label{def-phi-presheaf}
Let $i\colon Z\to S$ be a closed immersion and let $U:=S\setminus Z$ be its open complement. Given a pair $(X,t)$ consisting of $X\in\SmRig_S$ and a partial section $t\colon Z\to X$, i.e.\ inducing a section $t\colon Z\to X_Z := X\times_SZ$, we define the presheaf
\[ \Phi_S(X,t) := (X\sqcup_{X_U}U) \times_{i_*X_Z} S. \]
More explicitly, evaluating at $Y\in\SmRig_S$ we have that
\begin{align*}
\Phi_S(X,t)(Y) = \begin{cases} \Hom_S(Y,X)\times_{\Hom_Z(Y_Z,X_Z)}* & (Y_Z\neq\varnothing) \\ * & (Y_Z=\varnothing) \end{cases}
\end{align*}
where $\Hom_Z(Y_Z,X_Z)$ is pointed at the map $Y_Z \to Z \to[t] X_Z$. 
\end{defi}

\begin{rem}
We get a functorial assignment $(X,t) \mapsto \Phi_S(X,t)$ and for every morphism $f\colon T\to S$ there is a natural map 
\[ f^*\Phi_S(X,t) \To \Phi_T(X_T,t_T) \]
which is an isomorphism when $f$ is smooth.
\end{rem}

\begin{lem} \label{lem-phi-presheaf-etale}
Let $p\colon X'\to X$ be an \'etale morphism in $\SmRig_S$ and let $t\colon Z\to X$ and $t'\colon Z\to X'$ be partial sections which are compatible, i.e.\ $t=p\circ t'$. Then the induced map
\[ \Phi_S(p) \colon \Phi_S(X',t') \To \Phi_S(X,t) \]
is a Nisnevich equivalence.
\end{lem}
\begin{proof}
This is \emph{\'{E}tape 2} in the proof of \cite[Prop.~1.4.21]{AyoubRig}.
\end{proof}

\begin{lem} \label{lem-phi-presheaf-contractible}
Let $n\geq 1$ and let $\sigma\colon Z \inj S \to \Anrig_S$ be the restriction of the zero section of the canonical morphism $\pi\colon\Anrig_S\to S$. Then the canonical map $\Phi_S(\Anrig,\sigma) \to S$ is an $\Arig$-equivalence.
\end{lem}
\begin{proof}
A nullhomotopy is given by the map
\[ \Arig \times \Phi_S(\Anrig,t) \To \Phi_S(\Anrig,t), \quad (a,f) \mapsto af. \]
\end{proof}

\begin{thm}[Gluing]
\label{thm-gluing}
Let $i\colon Z\inj S$ be a closed immersion of rigid spaces with open complement $j\colon U\inj S$. Then for every $F\in\RigH(S)$ the square
\carremaptag{$\heartsuit$}{j_\#j^*F}{\epsilon}{}{F}{\eta}{U}{}{i_*i^*F}
is a pushout square (where the maps without labels are the unique ones).
\end{thm}
\begin{proof}
The category $\RigH(S)$ is generated by representables under weakly contractible colimits (cf. Remark~\ref{rem.Pes.gen}) and all the functors in the square $(\heartsuit)$ preserve these colimits, either because they are left-adjoint or by Corollary~\ref{cor-i_*-preserves-colimits}. Hence we may assume that $F=\Lmot X$ for some $X\in\SmRig_S$ so that it suffices to show that the canonical map
\[ X \sqcup_{X_U}U \To i_*X_Z \]
in $\PSh(\SmRig_S)$ is a motivic equivalence. 
We prove more generally that is a motivic equivalence in $\PSh(\Rig_S$).\footnote{At this point there is a mistake in the proof of  \cite[Thm.~4.18]{Hoy-EQ} because the category of smooth schemes (over a fixed base) does not have fibre products; see also the following footnote. This has been fixed in the latest version on the arXiv \cite[Thm.~4.18]{Hoy-EG-arxiv-v5}. We thank Marc Hoyois for communicating to us that the same problem appeared in a previous version of our article.}
The latter can be checked after base change along all maps $Y\to i_*X_Z$ for $Y\in\Rig_S$, i.e. it suffices to show that the maps
\[ \tag{$\spadesuit$}  (X \sqcup_{X_U}U) \times_{i_*X_Z} Y \To Y \]
are equivalences. We claim that we can reduce to the case where $Y=S$.

Indeed, given $p\colon Y\to S$ in $\SmRig_S$ and denoting by $q\colon Y_Z \to Z$ and $k\colon Y_Z\inj Y$ base changes of $p$ and $i$, respectively, we see that the map $(\spadesuit)$ identifies with the map
\[ p_\#\bigl( (p^*X \sqcup_{p^*X_U}p^*U) \times_{k_*(p^*X)_{Y_Z}} Y\bigr) \to p_\#Y \]
by a relative version of the projection formula\footnote{More precisely, for a morphism $f\colon T\to S$, presheaves $E,F\in\PSh(\Rig_S)$, and $G\in\PSh(\Rig_T)$ one has that $f_\#(f^*E\times_{f^*F}G) \simeq E\times_F G$. The proof reduces to the case of representable presheaves and then uses the existence of fibre products in $\Rig_S$.} and since we have equivalences
\[ p^*i_*X_Z \simeq k_*q^*X_Z \simeq k_*q^*i^*X \simeq k_*k^*p^*X \simeq k_*(p^*X)_{Y_Z} \]
where the first one is smooth base change (Proposition~\ref{prop-smooth-BC-and-PF}).

By adjunction, the datum of a map $S\to i_*X_Z$ corresponds to a section $t\colon Z\to X_Z$. Translating into the notation from Definition~\ref{def-phi-presheaf}, we have thus reduced to having to show that the canonical map $\Phi_S(X,t) \to S$ is a motivic equivalence for all $(X,t)$ as in loc.\! cit.

By Nisnevich descent (Proposition~\ref{prop-nisnevich-descent-RigH}) we may assume that $S$ is affine and always can replace by an open subspace. Thus we find an open neighbourhood $U$ of $Z$ in $S$, an open neighbourhood $V$ of $t(Z)$ in $X$, and an isomorphism $V\cong\Bnrig_U$ for a suitable $n\geq 1$ such that $t$ identifies with the partial zero section under this isomorphism \cite[Prop.~1.3.16]{AGV}. Applying Lemma~\ref{lem-phi-presheaf-etale} to the morphism $\Bnrig\cong V\inj X$ we obtain a Nisnevich equivalence $\Phi_S(X,t)\to[\simeq]\Phi_S(\Bnrig,\sigma)$. Applying the same lemma to the open immersion $\Bnrig\inj\Anrig$ we get a Nisnevich equivalence $\Phi_S(\Bnrig,\sigma)\to[\simeq]\Phi_S(\Anrig,\sigma)$. By Lemma~\ref{lem-phi-presheaf-contractible}, the presheaf $\Phi_S(\Anrig,\sigma)$ is $\Arig$-contractible.
\end{proof}

\begin{cor}
\label{cor-closed-im-ff}
    Let $i\colon Z\rightarrow S$ be a closed immersion. Then the functor $i_{*}\colon \RigH(Z)\to \RigH(S)$ is fully faithful. 
\end{cor}
\begin{proof}
    This is standard\footnote{Proofs in other situations can be found in \cite[1.4.23]{AyoubRig}, \cite[Cor. 4.19]{Hoy-EQ}, or \cite[Cor. 2.2.2 (1)]{AGV}.} and follows from Theorem~\ref{thm-gluing}, smooth base change (Proposition~\ref{prop-smooth-BC-and-PF}), and Nisnevich descent. Nevertheless, let us give a proof for completion.

    We have to show that for any $F\in \RigH(Z)$ the counit $i^{*}i_{*}F\rightarrow F$ is an equivalence. Using Theorem~\ref{thm-gluing}, it is enough to show that $i_{*}$ is conservative. In particular, it is enough to show that for $f\colon F\rightarrow G$ in $\RigH(Z)$ such that $i_{*}f$ is an equivalence, we have that $f$ is an equivalence. By Remark~\ref{rem.Pes.gen} it is suffices to test that $f$ is an equivalence at any $X\in \SmRig_{Z}$. The proof of \cite[Lem. 2.2.5 (2)]{AGV} shows that we may assume that there exists an $S_{X}\in \SmRig_{S}$ such that $X = S_{X}\times_{S} Z$. In particular, $f(X)$ is equivalent to $i_{*}f(S_X)$.
\end{proof}

\begin{cor}
    Let us consider a cartesian square
    \[ 
    \begin{tikzcd}
			T_{Z}\arrow[r,"i'",hookrightarrow]\arrow[d,"f'"]& T\arrow[d,"f"]\\
			Z\arrow[r,"i",hookrightarrow]&S
	\end{tikzcd}
    \]
    where $i$ is a closed immersion. Then the exchange transformation $f^{*}i_{*}\to i'_{*}f'^{*}$ is an equivalence.
\end{cor}
\begin{proof}
    This follows from Corollary~\ref{cor-closed-im-ff} and the smooth base change of Proposition~\ref{prop-smooth-BC-and-PF}.
\end{proof}

\section{Stable  \texorpdfstring{$\AA^{1}$}~-homotopy theory in rigid geometry}
\vspace{6pt}

 \noindent In this section we want to define the stable $\Arig$-homotopy category. To do so, we first want to show that the effective version satisfies all the functorial properties as $\RigH$. Afterward, we $\otimes$-invert the pointed projective line. This will yield a stable homotopy functor in the sense of Ayoub \cite{AyoubThesis}. In particular, using the results of op.\!\! cit., we will obtain a six-functor formalism with respect to algebraic maps. 

\vspace{6pt}\noindent\textbf{Notation.} In this section, let $S$ be a rigid space and let $\Vcal\in\CAlgPr$. For results using the analytification functor, let $R$ be an adic ring with ideal of definition $I$ and set $B:=\Spec(R)\setminus\textup{V}(I)$ and $B^\an=\textup{Spf}(R)^\textup{rig}$.

\subsection{Change of coefficients}
\label{sec.change.of.coeff.monoidal}

In Section \ref{sec.unstable}, we constructed the rigid $\Arig$-homotopy category via $\Arig$-localisation of Nisnevich sheaves of spaces. In this short subsection, we want to explain how we can extend our results to $\Arig$-invariant Nisnevich sheaves with $\Vcal$-valued coefficients.

\begin{rem}
\label{rem-change-of-coeff}
	We have $\Sh_{\Nis}(\SmRig_{S})\otimes\Vcal\simeq \Sh_{\Nis}(\SmRig_{S},\Vcal)$ \cite[Prop. 2.4]{DrewMHM}. This is also compatible with $\Arig$-localisation, hence
	\[
        \RigH(S)\otimes\Vcal\simeq\RigH(S,\Vcal).
	\]
    Furthermore, by Remark~\ref{rem-RigH-monoidal} the category $\RigH(S)$ is a commutative algebra object in $\PrLO$. Thus, the same is true for $\RigH(X)\otimes\Vcal$. In particular, the $\otimes$-product in $\RigH(S)\otimes\Vcal$ is symmetric and commutes with arbitrary colimits in both variables.
\end{rem}

With the above discussion it is not hard to see that the base change of a pullback formalism is again a pullback formalism \cite[\S 8]{DrewMHM}. As our construction is not directly covered by loc.\!\! cit., let us be more precise and give a statement with proof for completion.

\begin{lem}
\label{lem-effective-coeff}
	 The induced functor 
        $$
        \RigH\otimes \Vcal \colon\Rig\op\to \CAlgPr_{/\Vcal},\quad S\mapsto \RigH(S)\otimes \Vcal,\quad f\mapsto f^{*}\otimes\id_{\Vcal}
        $$ 
        has the following properties.
	\begin{enumerate}
		\item[(1)] For any smooth morphism $f$ inside $\Rig$ the pullback $f^{*}\otimes \id_{\Vcal}$ admits a left-adjoint $f_{\sharp}\otimes \id_{\Vcal}$ satisfying smooth base change and projection formula (cf. Proposition~\ref{prop-smooth-BC-and-PF}). 
		\item[(2)] For each $S\in\Rig$ and natural projection $p\colon \Arig_{S}\rightarrow X$ the functor $p^{*}\otimes \id_{\Vcal}$ is fully faithful.
	\end{enumerate}
\end{lem}
\begin{proof}
The proof is similar to \cite[Prop. 8.5]{DrewMHM} but let us give an argument nonetheless.\par
	If $f$ is smooth, then it admits a left-adjoint as this is true for $\RigH$ \cite[Lem. 8.4]{DrewMHM}. The main argument is that $\RigH(S)\otimes\Vcal\simeq \Fun^{R}(\RigH(S)\op,\Vcal)$ by \cite[Prop 4.8.1.17]{HA}, where $\Fun^{R}$ denotes those functors that admit a left-adjoint. Composing with $f^{*,\opp}\leftrightarrows f_{\sharp}\op$ for $\RigH\op$ yields the desired adjunction. This construction allows us to interpret smooth base change and the projection formula as composing with the transpose of the exchange transformation resp. the projection formula in $\RigH$, which is an equivalence by Proposition~\ref{prop-smooth-BC-and-PF}. This shows (1) and analogously also (2). 
\end{proof}

\subsection{Effective motives and localisation}

\vspace{6pt}\noindent\textbf{Notation.} 
For the rest of this section, we will assume that $\Vcal$ is furthermore stable.

\begin{defi}
	We define the \textit{effective stable homotopy category with coefficients $\Vcal$} by 
	\[
		\RigSH^{\eff}(S,\Vcal)\coloneqq\RigH(S)\otimes \Vcal.
	\]
	We denote the induced functor by $\RigSH^{\eff}_{\Vcal}$ and if $\Vcal\simeq \Sp$, then we simply write $\RigSH^\eff(S)$ and $\RigSH^{\eff}$. 
\end{defi}

\begin{rem}
	By Remark~\ref{rem-change-of-coeff}, we see that $\RigSH^{\eff}_\Vcal$ satisfies Nisnevich descent. Further, by \cite[Ex. 4.8.1.23]{HA} $\RigSH^{\eff}(S)$ is equivalent to the stabilisation of $\RigH(S)$.
\end{rem}

\begin{notation}
	Let $f\colon T\rightarrow S$ be a morphism of rigid spaces and $\Vcal$ a symmetric monoidal stable presentable category. Abusing notation, we will simply write 
    $$
    f^{*}\colon \RigSH^{\eff}(S,\Vcal)\to \RigSH^{\eff}(T,\Vcal)
    $$
    instead of $f^{*}\otimes \id_{\Vcal}$ and similarly $f_{*}$ for its right-adjoint resp. $f_{\sharp}$ for its left-adjoint (if it exists).
\end{notation}

\begin{rem}
	 We have seen in Remark~\ref{rem.closed.no-colim} that for closed immersions $i\colon Z\hookrightarrow S$ the pushforward $i_{*}$ on $\RigH$ does not preserve colimits in general. The obstruction is also very clear, in $\RigH$ the initial and final object do not agree. But $i_{*}$ on $\RigSH^{\eff}(S,\Vcal)$ does certainly preserve the initial object, as this obstruction vanishes. To be more precise, the inclusion $\iota\colon \PSh_\emptyset(\SmRig_S,\Vcal)\hookrightarrow \PSh(\SmRig_S,\Vcal)$ admits a right-adjoint $r$ given by $F\mapsto F\times_{F(\emptyset/S)} 0$ thus also $i_{\emptyset*}$ admits a right-adjoint, or equivalently commutes with all colimits. Thus by Lemma \ref{lem.i.compatible} the functor $i_*$ commutes with colimits.  
\end{rem}

\begin{notation}
	Let $i\colon Z\rightarrow S$ be a closed immersion. Then we denote the right-adjoint of $i_{*}$ by $i^{!}$.
\end{notation}

From the gluing in Theorem~\ref{thm-gluing} we immediately obtain a localisation sequence.
 
\begin{prop}[Localisation]
\label{prop-localisation}
    Let $i\colon Z\inj S$ be a closed immersion with open complement $j\colon U\inj S$. Then for any $F\in\RigSH^{\eff}(S,\Vcal)$, we have fibre sequences
    \[
    \begin{tikzcd}[row sep = 0.3em, column sep = 1.7em]
        j_{\sharp}j^{*}F\arrow[r,""] & F \arrow[r,""]&i_{*}i^{*}F,\\
        i_{*}i^{!}F\arrow[r,""] &F\arrow[r,""] & j_{*}j^{*}F.
    \end{tikzcd}
    \]
\end{prop}
\begin{proof}
    The statement of the lemma is equivalent to the following statement. For any closed immersion $i\colon Z\hookrightarrow S$ in $\Rig$ with open complement $j\colon U\hookrightarrow S$, the square 
		\[
			\begin{tikzcd}
				\RigH(U)\otimes \Vcal\arrow[r,"j_{\sharp}\otimes\id_{\Vcal}"]\arrow[d,""]& \RigH(S)\otimes \Vcal\arrow[d,"i^{*}\otimes\id_{\Vcal}"]\\
				\ast\arrow[r,""]&\RigH(Z)\otimes \Vcal
			\end{tikzcd}
		\]
		is a pushout\footnote{Note that the equivalence $\PrL \simeq (\Pr^{\mathrm{R}})^{\op}$ maps pushout squares to pullback squares \cite[Cor. 5.5.3.4]{HTT}.} in $\PrL$ \cite[Prop.~9.4.20]{Robalo-Thesis}.
        In particular, as the functor $-\otimes \Vcal$ preserves colimits, we can use the same arguments as in the proof of Lemma~\ref{lem-effective-coeff} to assume $\Vcal\simeq \Sp$. In view of Remark~\ref{rem-change-of-coeff}, we have $\RigSH^{\eff}(-)\simeq \Sp(\RigH(-,\Spc_{*}))$, where $\Spc_{*}$ denotes the category of pointed spaces. We have an equivalence 
        $$ \Sp(\RigH(-,\Spc_{*}))\simeq \lim(\dots\xrightarrow{\Omega}\RigH(-,\Spc_{*}))
        $$
        inside $\Pr^{R}$ \cite[Prop. 1.4.2.24]{HA}. In particular, the second sequence of the proposition is a fibre sequence if and only if it is so in $\RigH(-,\Spc_{*})$. Since the first sequence is the adjoint of the second one, we may assume $\Vcal\simeq \Spc_*$.

        We will show that the first sequence is a cofibre sequence in $\RigH(-,\Spc_{*})$, which proves that the latter one is a fibre sequence. For this let us note that the forgetful functor $p\colon \Spc_*\rightarrow \Spc$ takes pushout squares to pushout squares and reflects them \cite[Prop. 4.4.2.9]{HTT}. As in the proof of \cite[Prop. 5.2]{Hoy-EQ}  let us look at the following two squares
        $$
         \begin{tikzcd}
            j_{\sharp}j^*p(F)\arrow[r,""]\arrow[d,""]&j_{\sharp}j^*p(F)\coprod_{U} S\arrow[r,""]\arrow[d,""]& p(F)\arrow[d,""]\\
				U\arrow[r,""]&S\arrow[r,""]&i_*i^*p(F).
		\end{tikzcd}
        $$
        The outer square is a pushout by Theorem~\ref{thm-gluing} (note that $F\in \RigSH^{\eff}(S,\Vcal)$ comes equipped with a zero section). The left square is by design a pushout. So the right square is also a pushout, which is also the image of the first sequence of the proposition under $p$.
\end{proof}

\subsection{Thom motives and stabilisation}
\label{sec.thom-motives-stabilisation}

 For any vector bundle $p\colon V\rightarrow S$ with zero section $s\colon S\rightarrow V$, we can define the suspension $\Sigma_{V}\coloneqq p_{\sharp}s_{*}$ and loop $\Omega_{V}\coloneqq s^{!}p^{*}$. Note that this yields an adjunction
\[
	 \begin{tikzcd}
		\Sigma_{V}\colon\RigSH^{\eff}(S,\Vcal)\arrow[r,"",shift left = 0.3em]&\arrow[l,"",shift left = 0.3em]\RigSH^{\eff}(S,\Vcal)\colon \Omega_{V}.
	\end{tikzcd}
\]  

\begin{lem}
\label{lem-thom}
     Then the following are equivalent.
    \begin{enumerate}
        \item[(i)] For any smooth morphism $f\colon T\rightarrow S$ of rigid spaces that admits a section $s$, the Thom transformation $\textup{Th}(f,s)\coloneqq f_{\sharp}s_{*}$ is an equivalence.
        \item[(ii)] For any vector bundle $V$ on $S$ the suspension of the unit $\Sigma_{V}1_V$ is $\otimes$-invertible inside $\RigSH^{\eff}(S,\Vcal)$.
        \item[(iii)] The suspension of the unit of the closed unit disc $\Sigma_{\BB^{1}_S}1_{\BB^{1}_{S}}$ is $\otimes$-invertible.
    \end{enumerate}
\end{lem}
\begin{proof}
    Equivalences $(i)\Leftrightarrow (ii)\Leftrightarrow (iii)$ are standard and follow from the fact that any smooth morphism $f\colon T\rightarrow S$ of rigid spaces factors locally as an \'etale morphism $T\rightarrow \BB^{n}_S$ followed by the projection \cite[\S 2.4]{CD1}.
\end{proof}

\begin{rem}
\label{Rem:Tate-twist}
    Let $V\rightarrow S$ be a vector bundle. The localisation sequence in Proposition~\ref{prop-localisation} implies that 
    $$
    \Sigma_{V}1_{V}\simeq \frac{V}{V\setminus \lbrace 0\rbrace},
    $$
    where the right-hand side denotes the cofibre of the natural open immersion $V\setminus\lbrace 0\rbrace \rightarrow V$ induced by the zero section.
    
    Further, note that we have a commutative diagram with pushout squares
     $$
         \begin{tikzcd}
            \Brig_{S}\setminus \lbrace 0\rbrace\arrow[r,""]\arrow[d,""]&\Arig_{S}\setminus \lbrace 0\rbrace\arrow[r,""]\arrow[d,""]& 0\arrow[d,""]\\
				\Brig_{S}\arrow[r,""]&\Arig_{S}\arrow[r,""]&\frac{\Arig_{S}}{\Arig_{S}\setminus \lbrace 0\rbrace}.
		\end{tikzcd}
        $$
        In particular, we have
        $$
           \frac{\Brig_{S}}{\Brig_{S}\setminus \lbrace 0\rbrace}\simeq \frac{\Arig_{S}}{\Arig_{S}\setminus \lbrace 0\rbrace}
        $$
        inside $\RigSH^{\eff}(S)$. Thus, condition (iii) in Lemma \ref{lem-thom} is equivalent to $\Sigma_{\Arig_S}1_{\Arig_S}$ being $\otimes$-invertible.

        Finally, let us look at the diagram
         $$
         \begin{tikzcd}
            \Arig_{S}\setminus \lbrace 0\rbrace\arrow[r,""]\arrow[d,""]&\Arig_{S}\arrow[r,""]\arrow[d,"\infty"]& 0\arrow[d,""]\\
				\Arig_{S}\arrow[r,"0"]&\PP^{1}_{S}\arrow[r,""]&(\PP^{1}_{S},\infty),
		\end{tikzcd}
        $$
        where $(\PP^{1},\infty)$ is defined as the pushout of the right square. All squares in the diagram are a pushout thus we have 
        $$
            (\PP^{1}_{S},\infty)\simeq  \frac{\Arig_{S}}{\Arig_{S}\setminus \lbrace 0\rbrace}.
        $$
\end{rem}

\begin{defi}
	We define the category of \textit{rigid analytic motivic spectra over $S$}
	\[
		\RigSH(S,\Vcal)\coloneqq\RigSH^{\eff}(S,\Vcal)[(\PP^{1}_{S},\infty)^{-1}].
	\]
\end{defi}

\begin{rem}
\label{rem-RigSH-mon}
    As remarked in \cite[\S 6.1]{Hoy-EQ}, the definition of inversion along an element yields for a morphism of rigid spaces $f\colon S\rightarrow T$ an equivalence
    $$
        \RigSH^{\eff}(S,\Vcal) \otimes_{\RigSH^{\eff}(S,\Vcal)}\RigSH(T,\Vcal)\simeq \RigSH(S,\Vcal)
    $$
    via $f^*$ as symmetric monoidal categories. In particular, the argumentation of the proof of Lemma \ref{lem-effective-coeff} shows that $\RigSH(-,\Vcal)$ again satisfies projection formula, smooth base change, and localisation. We refer to \cite[Prop. 8.5]{DrewMHM} for a rigorous proof of this statement. Furthermore, by construction $(\PP^{1}_{S},\infty)$ is $\otimes$-invertible inside $\RigSH(S,\Vcal)$.

    Lastly, let us remark that by design $\RigSH(S,\Vcal)$ is a commutative $\Vcal$-algebra in $\PrL$ \cite[Def. 2.6]{Robalo-Thesis} as $\RigSH^{\eff}(S,\Vcal)\in \CAlgPr_{/\Vcal}$ by Remark \ref{rem-change-of-coeff}.
\end{rem}

Using this definition, we can conclude that $X\mapsto \RigSH(X^{\an})$, for any $B$-scheme $X$, is a stable homotopy functor \cite{AyoubThesis}  (see also \cite[Prop. 5.11]{DrewMHM}), i.e. it satisfies the properties (1)-(6) listed in \cite[\S 1.4.1]{AyoubThesis}.\par
In particular, \cite[Scholie 1.4.2]{AyoubThesis} implies the existence of a six-functor formalism on the restriction of $\RigSH$ to quasi-projective $K$-schemes. Using the results of Cisinksi-Deglise, we can extend this to all separated finite type $B$-schemes \cite[Thm. 2.4.50]{CD1}.

\begin{thm}
\label{thm-6ff}
The functor\footnote{The functoriality follows similarly as in \cite[\S 9.1]{RobaloDiss}.} 
\begin{align*}
	\RigSH_{\Vcal}\colon (\Sch^{\textup{sep,ft}}_{B})\op\rightarrow \CAlgPr_{/\Vcal},\quad X\mapsto \RigSH(X^{\an},\Vcal),\quad f\mapsto f^{*}
\end{align*}
satisfies Nisnevich descent and extends to a full six-functor formalism in the sense of Liu--Zheng \cite{Liu-Zheng}. Further, by passing to homotopy categories this functor defines a motivic triangulated functor in the sense of Cisinski--Déglise \cite{CD1}.
\end{thm}
\begin{proof}
    Nisnevich descent follows from Proposition \ref{prop-nisnevich-descent-RigH}.

    By \cite[Prop. 5.11]{DrewMHM} together with Remark \ref{rem-RigH-monoidal}, Lemma \ref{lem-effective-coeff}, Proposition \ref{prop-localisation}, and Lemma \ref{lem-thom} the composition of $\RigSH_{\Vcal}$ with the homotopy functor yields a stable homotopy $2$-functor in the sense of \cite[\S 1.4.1]{AyoubThesis}. Thus, using that any finite type separated $B$-scheme admits a Nagata compactification \cite[Thm. 4.1]{NagataComp}, we can apply \cite[Thm. 2.4.26]{CD1} and \cite[Prop A.5.10]{Mann-Thesis} to obtain a pre-six-functor formalism in the sense of \cite{Mann-Thesis}. To lift this to a full six-functor formalism, we are left to show that for any proper map of separated finite type $B$-schemes $f\colon X\rightarrow Y$ the functor $f_*$ admits a right-adjoint \cite[Prop. A.5.10]{Mann-Thesis}. 
    
    Let us assume that $\Vcal\simeq \Sp$, then this follows from \cite[Prop. 5.17]{DrewMHM}. Now let $\Wcal$ be a symmetric monoidal stable presentable category. Since $f_{*}$ is a morphism inside $\Pr^{L}$, so is the right-adjoint $f_{*}\otimes \id_{\Wcal}$ of $f^{*}\otimes \id_{\Wcal}$ \cite[Lem. 8.4]{DrewMHM}. In particular, $f_{*}\otimes \id_{\Wcal}$ admits a right-adjoint.
\end{proof}

\begin{rem}[A note on hypercompleteness and change of topology]
    Let $\tau$ denote either the Nisnevich or \'etale toplogy.
    Consider the functor 
    \[
        \RigSH_{\tau,\Vcal}^{(\wedge)}\colon X\mapsto \Sh^{\AA^1(,\wedge)}_{\tau}(\SmRig_X,\Vcal)[(\PP^1_S,\infty)^{-1}],\quad f\mapsto f^*.
    \]
    Then $\RigSH_{\tau,\Vcal}^{\AA^1(,\wedge)}$ satisfies the axioms of \cite[\S 1.4.1]{AyoubThesis}. For consistency, we write $\RigSH^{\wedge}_{\tau}(X,\Vcal)\coloneqq \RigSH_{\tau,\Vcal}^{(\wedge)}(X)$.

     We will not prove this result in its entirety as most of it is just following verbatim the proof of Theorem \ref{thm-6ff}.  For $\tau=\et$, we can follow the proof in its entirety. In the hypercomplete case, the only instance where we have to be careful is in the proof of Lemma \ref{lem.i.compatible}. We have to replace the sheafification in the proof by hypersheafification. However, the formulas stay the same \cite[Cor. 6.5.3.13]{HTT} --- see also the beginning of \cite[\S 6.5.3]{HTT}.
\end{rem}

The reason we care about hypercompleteness is the following result. 

\begin{prop} \label{Prop:hypersheaves-derived-category}
    Let $\Lambda$ be a ring. Let $(\Ccal,\tau)$ be a site whose underlying category $\Ccal$ is discrete. Assume that $\Ccal$ has enough points. Then the canonical functor $\mathrm{Ch}(\Sh_\tau(\Ccal,\Mod(\Lambda))) \to \PSh(\Ccal,\Dcal(\Lambda))$ induces a t-exact equivalence
    \[ 
        \Dcal_\tau(\Ccal,\Lambda)  \To[\simeq] \Sh_{\tau}^{\wedge}(\Ccal,\Dcal(\Lambda)).
    \]
\end{prop}
\begin{proof}
    This works similarly as \cite[Prop. 7.1]{Scholze6FF}.
\end{proof}

In particular, we have $\Dcal(\Sh_\tau(\SmRig_X,\Mod(\Lambda)))\simeq \Sh_\tau^\wedge(\SmRig,\Dcal(\Lambda)).$ This will play a role later, when we prove rigidity in Section~\ref{sec:rigidity}.

\subsection{The relation with $\PP^1$-spectra}
\label{sec-RigSH-as-spectra}

In this subsection, we shortly want to relate our definition of the stable motivic homotopy category with the more classical approach via $\PP^1$-spectra. This will be an easy consequence from our constructions and the analytification functor
$$
    (-)^{\an}\colon \textup{SH}^{\eff}(B)\to \RigSH^{\eff}(B^\an).
$$

Let us first recall what we mean with $\PP^1$-spectra following Hovey \cite{Hovey}.\linebreak A slightly more thorough discussion can be found in \cite[\S 2.3]{RobaloDiss}. Let $\Ccal$ be a combinatorial simplicial symmetric monoidal model category, i.e. a presentable symmetric monoidal category. For any (cofibrant) object $X\in\Ccal$, in fact for any left Quillen endofunctor, one can define the model category of $X$-spectra in $\Ccal$, denoted by $\Sp^{\NN}(\Ccal, X)$. The underlying $1$-category consists of \textit{$X$-spectra} $F$, which are sequences $(F_0,F_1,\ldots)$ in $\Ccal$ together with morphisms $X\otimes F_n\rightarrow F_{n+1}$ for all $n\in \NN_0$. The notion of morphism of $X$-spectra is the obvious one. We endow $\Sp^{\NN}(\Ccal,X)$ with the stable model structure \cite[\S 3]{Hovey}. Let us remark that in general $\Sp^{\NN}(\Ccal,X)$ is not symmetric monoidal. For this reason one usually passes to \textit{symmetric spectra} together with the stable model structure, denoted by $\Sp^\Sigma(\Ccal,X)$. We do not want to go more into detail concerning $\Sp^\Sigma(\Ccal,X)$ and refer to \cite[\S 7]{Hovey}. It is important to remark that the associated categories via the simplicial nerve of $\Sp^\Sigma(\Ccal,X)$ and $\Sp^\NN(\Ccal,X)$ are not equivalent in general. However, as discussed in \cite[\S 10]{Hovey}, one can see that the only obstruction is the cyclic permutation $(123)\in \Sigma_3$ self map of $X^{\otimes 3}\coloneqq X\otimes X\otimes X$. To be more precise, Hovey shows that there is a zig-zag of Quillen equivalences between $\Sp^\Sigma(\Ccal,X)$ and $\Sp^\NN(\Ccal,X)$ if the cyclic permutation on $X\otimes X\otimes X$ is homotopic to the identity. In this case $X$ is called \textit{symmetric}. Moreover, Robalo shows that if $X$ is symmetric, we have an equivalence 
$$
    \Ccal[X^{-1}]\To[\simeq] \Sp^\Sigma(\Ccal,X)
$$
as presentable symmetric monoidal categories \cite[Thm. 2.26]{Robalo-Thesis}.

Now let us come back to our situation and show how we can interpret $\RigSH$ as $\PP^1$-spectra. For this consider the category $\RigH(B^\an)^{\wedge}_*\coloneqq \RigH(B^\an,\Spc_*)$ with the induced $\otimes$-product via Remark \ref{rem-change-of-coeff}. We also denote by $\Gm\coloneqq \Arig\setminus\lbrace 0\rbrace$ the analytic affine line without $0$.

\begin{lem}
\label{lem-P1-spectrum}
    The pointed projective line $(\PP^1,\infty)\in \RigH(B^\an)^{\wedge}_*$ is symmetric and equivalent to $S^1\otimes (\Gm,1)$.

    In particular, we have 
    $$
        \RigSH(B^{\an})\simeq \Sp^\NN(\RigH(B^\an)^{\wedge}_*,(\PP^1,\infty)).
    $$
\end{lem}
\begin{proof}
    Combining our discussion in Section \ref{sec.change.of.coeff.monoidal} and Lemma \ref{analytification-motivic-spaces--lem}, we see that  $(-)^{\an}$ is a monoidal functor. The cyclic permutation of $(\PP^1,\infty)^{\otimes 3}\in  \textup{H}(B)^\wedge_*$ is equivalent to the identity \cite[Lem. 4.4]{voevodsky-icm-1998}. In particular, the same is true for $(\PP^{1},\infty)^{\an}$. As the analytification functor preserves colimits, we see that $(\PP^{1},\infty)^{\an}\simeq (\PP^1,\infty)\in \RigH(B^\an)^{\wedge}_*$.

    Moreover, by the same argument we have 
    $$
        (\PP^1,\infty)\simeq S^1\otimes(\Gm,1)\in \RigH(B^\an)^{\wedge}_*
    $$
    \cite[Lem. 2.15]{MV1}. Therefore, we have 
$$
    \RigSH(B^\an)\simeq \RigH(B^\an)^{\wedge}_*[(\PP^1,\infty)^{-1}],
 $$
as presentable symmetric monoidal categories.
\end{proof}

\begin{cor}
    There is an equivalence of presentable symmetric monoidal categories 
    $$
        \RigSH(B^{\an},\Vcal)\simeq \Sp^\NN(\RigSH^{\eff}(B^\an,\Vcal),(\PP^1,\infty)).
    $$
\end{cor}
\begin{proof}
    This follows from Lemma \ref{lem-P1-spectrum}, the monoidality in Section \ref{sec.change.of.coeff.monoidal} and the discussion in the beginning of this subsection.
\end{proof}

\section{Rigidity for $\AA^1$-motives}
\label{sec:rigidity}
As an application of our theory, we want to establish a rigidity result, similar to Bam\-bozzi--Vezzani \cite{BaVe} and Ayoub--Gallauer--Vezzani \cite{AGV}. The benefit of the $\AA^1$-invariant theory is that we obtain also a mod $p$ rigidity for rigid spaces over $\QQ_p$, whereas the references only show mod $\ell$ rigidity for $\ell\neq p$. This is a major improvement that surprisingly does not need any new techniques but the six-functor formalism established in Theorem \ref{thm-6ff}. The main observation is that if $K$ is a nonarchimedean field in mixed characteristic $(0,p)$, then \'etale cohomology with mod $p^n$-coefficients is in fact $\AA^1$-invariant \cite{Berkovich}. We will follow \cite{AGV} closely and elaborate at critical points, where one has to be careful in the mod $p$ setting. Moreover, we establish also in the mod $\ell$ setting an equivalence between the $\AA^1$-invariant and the $\BB^1$-invariant theory. 

Contrary to \cite{AGV}, we have to work solely with hypercomplete \'etale sheaves as we cannot prove continuity for $\AA^1$-motives at the moment. Moreover, we cannot establish a rational $p$-adic realisation functor, only an integral version. This is due to the fact that $\RigDA_{\et}^{\AA^1}$ is not compactly generated since $\AA^1$ is not quasi-compact. At the moment, we do not know how to circumvent this problem.

\begin{assumption}
Throughout this section let us fix an integer $N$ and let $\Lambda$ be a $\ZZ/N\ZZ$-module. Moreover, let $K$ be a nonarchimedean field of residue characteristic $p\geq 0$ and we assume that  $\mathrm{char}(K) = 0$ or that $p\nmid N$. Moreover, we will denote by $X$ a rigid space over $K$.

\end{assumption}

 \begin{notation}
     To stay consistent with the literature and differentiate between $\AA^1$- and $\BB^1$-homotopy theory, we denote 
     \[
         \RigDA_{\et}^{\AA^1,\wedge,\eff}(X,\Lambda) \coloneqq \Sh^{\AA^1,\wedge}_{\tau}(\SmRig_X,\Lambda)
     \]
     and by $\RigDA_{\et}^{\AA^1,\wedge}(X,\Lambda)$ its $\PP^1$-stabilisation. Similarly, for the $\BB^1$-invariant version.
 \end{notation}

\begin{defi}
    We say that a rigid space $X$ over $K$ is $\Lambda$-admissible if there exists an open covering $(X_i)_i$ of $X$ such that each $X_i$ has finite Krull dimension and finite punctual virtual $\Lambda$-cohomological dimension, cf.\ \cite[Def. 2.4.14]{AGV}.
\end{defi}

The main reason we need a condition on the local punctual virtual codimension is the following result.

\begin{lem}[\protect{\cite[Lem.~2.4.5]{AGV}}]
    Let $X$ be a $\Lambda$-admissible rigid space over $K$. Then Postnikov towers in $\Dcal(X_{\et},\Lambda)\simeq \Sh^{\wedge}(X_{\et},\Dcal(\Lambda))$ converge.
\end{lem}

The goal of this section is to prove the following rigidity result, relying on Theorem~\ref{thm-6ff}.

\begin{thm}(Rigidity)
\label{Thm:Rigidity}
    The natural functor
	\[
		\Dcal(X_{\et},\Lambda) \To[\iota^*] \RigDA_{\et}^{\AA^1,\wedge(,\eff)}(X,\Lambda)
	\]
	is an equivalence.
\end{thm}

We immediately obtain two important results. First, we obtain the following comparison between $\AA^1$-motives and $\BB^1$-motives.

\begin{cor}
\label{cor:comparison-with-B1-motives}
    Assume that $p\nmid N$. Then the fully faithful embedding
    \[
		\RigDA^{\BB^1,\wedge(,\eff)}_{\et}(X,\Lambda) \subseteq \RigDA^{\AA^1,\wedge(,\eff)}_{\et}(X,\Lambda)
	\]
    is an equivalence.
\end{cor}
\begin{proof}
    This follows from Theorem \ref{Thm:Rigidity} and the rigidity theorem in the $\BB^1$-case \cite[Thm.~2.10.3]{AGV}.
\end{proof}

\begin{cor}
\label{Thm:Real}
     Let $\ell$ be any prime. Assume that $\mathrm{char}(K) = 0$ or that $p\neq\ell$.
      Then there exists a symmetric monoidal realisation functor
	\[
		\hat{\rho}_\ell\colon \RigDA_{\et}^{\AA^1,\wedge}(X,\ZZ)\to \Dcal(X_{\et},\ZZ_{\ell})\coloneqq \lim_{n\geq 0}\Dcal(X_{\et},\ZZ/\ell^n).
	\]
\end{cor}
\begin{proof}
    The proof of the $\ell$-adic realisation, not given in this article, follows verbatim \cite[\S 5]{AyoubEt} using Theorem \ref{Thm:Rigidity}.
    Here we use the identification of $\Dcal(X_{\et}, \ZZ_\ell)$ with the diagram category of Ayoub; for the limit description see also \cite[\S 1-2]{Liu-Zheng-Adic}.
\end{proof}

While for $\ell\neq p$ this result is not new, we obtain a $p$-adic realisation functor if we work with $\AA^1$-motives. Moreover, the in the case $\ell\neq p$, we recover the $\ell$-adic realisation of \cite{BaVe}.
In the case $p\mid N$, we still obtain an adjunction between $\BB^1$-motives and $\AA^1$-motives.

\begin{prop}
    There is an adjunction
    $$
	 \begin{tikzcd}
		\LB \colon  \RigDAA(X,\Lambda)\arrow[r,"",shift left = 0.3em]&\arrow[l,"",shift left = 0.3em]\RigDAB(X,\Lambda)\colon j,
	\end{tikzcd}
	$$
    whose right-adjoint functor is fully faithful. In particular, the analytification functor of \cite{AyoubRig} is obtained by $\LB\circ \an^*$.
\end{prop}
\begin{proof}
    The category $\RigDABeff(X,\Lambda)$ is a full subcategory of $\RigDAAeff(X,\Lambda)$ since every $\BB^1$-invariant presheaf is $\AA^1$-invariant. Since the categories are presentable and $\BB^1$-invariant sheaves are closed under limits, the inclusion functor admits a left-adjoint $\LB$. This localisation passes to the stabilised categories \cite[Lem.~1.5.4]{AnnalaIwasa}.
\end{proof}

The proof of Theorem~\ref{Thm:Rigidity} will follow the approach of \cite{AGV}. To this end, we will work with the $\ell$-completed version of $\RigDA$, simply denoted by $\RigDA^{\AA^1}_{\et}(-,\Lambda)_{\lcomp}$, cf.\ \cite[\S 2.10]{AGV}.

\subsection{Some remarks on the assumption in Theorem~\ref{Thm:Rigidity}}

In Theorem~\ref{Thm:Rigidity}, we make assumptions on the characteristic of $K$ and the order of the coefficients. This is indeed a natural restriction that we cannot get around as the next observations will show. In particular, it is not possible to represent \'etale mod $p$ cohomology theories in equicharacteristic $p$.


\begin{prop}
	If $\mathrm{char}(K)=p>0$, then the category $\RigDA_{\et}(X,\ZZ)$ is $\ZZ[1/p]$-linear. 
\end{prop}
\begin{proof}
	The proof is as usual, cf. \cite[Lem.~3.10]{AyoubEt}. We use the fact that $\GG_{\aup}$ as a sheaf of abelian groups is equivalent to $0$. 
	Indeed, the proof is as in the classical case by using multiplication on the global sections
	$$
		\GG_{\aup}(X\times_{K}\AA^{1,\an}_{K}) = \Gamma(X,\Ocal_X)\hat{\otimes}_{K}K\lbrr T\rbrr,
	$$
	which contains the polynomial ring. Thus the map 
    \[
        \GG_{\aup}(X)\to  \Gamma(X,\Ocal_X)\hat{\otimes}_{K}K\lbrr T\rbrr,\ x\mapsto x\otimes T
    \]    
    yields a homotopy between the identity of $\GG_{\aup}(X)$ and the zero map. Now the Artin--Schreier sequence shows that $\ZZ/p\ZZ\simeq 0$ and hence  multiplication by $p$ equivalence in $\RigDA_{\et}(X,\ZZ)$.
\end{proof}

\begin{rem}
	Note that in the proof of the proposition above, we only use the Artin--Schreier sequence and $\GG_{a}$. If one has an Artin--Schreier sequence for $\GG_{\aup}^{+}$, then the same result holds again, for example in the tame topology \cite{HuebnerTame,HuebnerSchmidt}. Indeed, the same proof applies using that $\GG_{\aup}^{+}(\Arig_{X}) = \Gamma(X,\Ocal^{+}_{X})\hat{\otimes}_{\Ocal_{K}}\Ocal_{K} \lbrr T\rbrr$.
\end{rem}

\begin{rem}
    The same argument shows that one has to be careful with tilting in $\Arig$-homotopy theory.
	There cannot be a tilting equivalence between $\RigDA_{\et}^{\AA^1}(K,\ZZ)$ and $\RigDA_{\et}^{\AA^1}(K^{\flat},\ZZ)$ induced by the equivalence $K_{\et}\simeq K^{\flat}_{\et}$ on \'etale sites. Indeed, assume that such an equivalence exists. Since $\ZZ/p\ZZ$ is $\AA^{1}$-invariant as an \'etale sheaf, we see that $\ZZ/p\ZZ$ is contained naturally in $\RigDA^{\AA^1,\eff}_{\et}(K,\ZZ)$. However, by the tilting equivalence we see that $\ZZ/p\ZZ$ fits into the Artin--Schreier sequence
	$$
		0 \to \ZZ/p\ZZ\to \GG_{\aup,K^{\flat}}\to \GG_{\aup,K^{\flat}} \to 0.
	$$
	Now note that $\GG_{\aup,K^{\flat}}\simeq_{\AA^{1}} 0$ in $\RigDA_{\et}^{\AA^1}(K^{\flat})$ and thus also $\ZZ/p\ZZ \simeq_{\AA^{1}} 0$, contradicting the fact that $\Hup^{0}_{\et}(-,\ZZ/p\ZZ)$ is non-trivial in general.
\end{rem}

\subsection{Embedding Theorem}

In this subsection, we prove that $\iota^*$ is fully faithful by exploiting the Kummer sequence as in the classical setting. We work with the completed versions as in \cite[\S 2.10]{AGV}.

\begin{thm}
\label{Thm:Embedding-Theorem}
    	The functor 
	\[
		\iota^{*}\colon \Dcal(X_{\et},\Lambda)_{\lcomp}\to \RigDA_{\et}^{\AA^1}(X,\Lambda)_{\lcomp}
	\]
	is fully faithful.
\end{thm}

The first key step is the reduction to the effective case. This will be the content of the next two lemmas.

\begin{rem}[\protect{\cite{dJvdP}}]
	The group homomorphism $(-)^{N}\colon \Ocal^{\times}\to \Ocal^{\times}$ of \'etale group schemes is surjective. Thus we obtain the Kummer sequence
	\[
		1\to \mu_{N}\to \Ocal^{\times}\to \Ocal^{\times}\to 1.
	\]
\end{rem}

\begin{lem}
\label{Lem:Tate-Twist}
    There exists a $\otimes$-invertible object $\Lambda_{\ell}(1)$ in $\Sh^{\wedge}_{\et}(X,\Lambda)_{\lcomp}$ together with a  morphism 
	\[
		\Lambda_{\ell}\to \Lambda_\ell(1)[1]
	\]
	in $\Sh^{\wedge}_{\et}(\BB_{X}\setminus \lbrace 0\rbrace,\Lambda)_{\lcomp}$ and a trivialisation after pullback along the unit section $1_X\in (\BB_{X}\setminus \lbrace 0\rbrace)(X)$. Moreover, the induced morphism 
    \[
        \Bigl(\frac{\BB^1_X}{\BB^1_X\setminus 0}[1]\Bigr)^{\wedge}_{\ell} \To \iota_X^*\bigl(\Lambda_\ell(1)[1]\bigr)
    \]
    is an equivalence in $\RigDA_{\et}^{\AA^1,\eff}(X,\Lambda)_{\lcomp}$.
\end{lem}
\begin{proof}
    The construction is the same as \cite[Lem. 2.10.8]{AGV} by Remark~\ref{Rem:Tate-twist}. The last part follows as in loc.\!\! cit.\ from the compatibility of the analytification functor and $\iota^*$, cf. Lemma~\ref{analytification-functor--lem}.
\end{proof}

\begin{cor}
\label{Cor:stab-equiv}
    The functor $\Sigma_+^{\infty}\colon \RigDA^{\AA^1,\wedge,\eff}_{\et}(K,\Lambda)_{\lcomp}\to\RigDA^{\AA^1,\wedge}_{\et}(K,\Lambda)_{\lcomp} $ is an equivalence.
\end{cor}
\begin{proof}
    This follows from the construction of the $\ell$-completion, the interpretation of $\RigSH$ as $\PP^1$-spectra, see Section \ref{sec-RigSH-as-spectra}, and Lemma \ref{Lem:Tate-Twist}, cf. \cite[Cor. 2.10.9]{AGV}.
\end{proof}

\begin{proof}[Proof of Theorem \ref{Thm:Embedding-Theorem}]
    By Corollary \ref{Cor:stab-equiv}, it is enough to check the statement for effective motives. Then by devissage, it is enough to check that any (discrete) \'etale sheaf $F$ on $X$ with $\Lambda$-coefficients is $\AA^1_X$-invariant, cf. \cite[Lem. 2.10.10]{AGV}.
    This holds true when $p\mid N$ in the $\mathrm{char}(K)=0$ case by \cite[Lem. 2.2]{Berkovich} and by \cite[Thm. 6.0.2]{dJvdP} in the $p\nmid N$ case.
\end{proof}

\subsection{Proof of the Rigidity Theorem}

We are left to show that the functor $\iota^*$ of Theorem~\ref{Thm:Rigidity} is essentially surjective. this result will crucially use two facts that we established in this article in Theorem~\ref{thm-6ff}. 

First, there exists an analytification functor $\DA_{\et}(X,\Lambda)\xrightarrow{\an}\RigDA^{\AA^1,\wedge}_{\et}(X^\an,\Lambda)$ for any separated finite type $K$-scheme $X$. 
Secondly, for any morphism $f$ between separated finite type $K$-schemes, we have the exceptional functors $f_!^\an\dashv f^{\an!}$ and $f^{\an*}\dashv f_*^\an$ satisfying base change and projection formula. Moreover, these functors are compatible with the analytification functor $(-)^{\an}$ from the above.

The following argument follows verbatim the proof of \cite[Thm. 2.10.3]{AGV}. We will recall the proof, but only up to the point where we need Theorem \ref{thm-6ff} to make the dependence clear.

Let us set $T \coloneqq \bigl(\frac{\BB^1_X}{\BB^1_X\setminus 0}[1]\bigr)^{\wedge}_{\ell}$. By Corollary \ref{Cor:stab-equiv} it is enough to treat the $T$-stable case. Moreover, since the question is local, we may assume that $X=\Spf(A)^{\rig}$ is affine and set $U=\Spec(A[\pi^{-1}])$, where $\pi\in A$ is a generator of the ideal of definition. Note that $\RigDAA(X,\Lambda)_{\lcomp}$ is generated by colimits under $\textup{M}(\Spf(B)^{\rig})/\ell^n$, where $B$ is rig-\'etale adic over $A$ satisfying the conclusion of \cite[Prop. 1.3.15]{AGV} and $n\in \NN$. As in the proof of \cite[Thm. 2.10.3]{AGV}, we may assume that there exists a smooth affine $U$-scheme $S$ and an open immersion $\nu\colon V\to S^\an$ such that $\Omega_{S/U}$ is free. Fixing a projective compactification $j\colon S\to P$ over $U$, we obtain the following commutative diagram 
\[
    \begin{tikzcd}
        V\arrow[rr,"\bar\nu", bend left = 4em] \arrow[r,"\nu"]\arrow[rd,"g",swap]& S^\an\arrow[r,"j^\an"]\arrow[d,"f^\an"] & P^\an\arrow[dl,"p^\an"]\\
         & X, &
    \end{tikzcd}
\]
where $g,f$, and $p$ denote the structure maps.

We have $\textup{M}(V) \simeq  f_\sharp^\an \nu_\sharp \Lambda$. By Theorem \ref{thm-6ff}, we see that $\textup{M}(V)$ is equivalent, up to shift and twist, to $f_!^\an\nu_\sharp\Lambda\simeq p^\an_*\bar\nu_\sharp\Lambda$. 

By Lemma \ref{Lem:Tate-Twist}, the map $\iota^*$ is stable under shifts and twists in the $T$-stable case. Thus, it remains to show that $p^\an_*\bar\nu_\sharp\Lambda/\ell^n$ is contained in its essential image.

From here on the rest is the same as in \cite{AGV}.


\bibliographystyle{alphaurl}
\bibliography{RigidSH}

\bigskip
  \footnotesize

  Christian Dahlhausen:\\  \textsc{Institut für Mathematik, Universtität Heidelberg, Im Neuenheimer Feld 205, 69120 Heidelberg, Germany}\par\nopagebreak
  \textit{E-mail address:}  \texttt{cdahlhausen@mathi.uni-heidelberg.de}

  \medskip

  Can Yaylali:\\ \textsc{Fakultät für Mathematik, Universität Duisburg-Essen, Thea-Leymann-Str. 9, 45127 Essen, Germany}\par\nopagebreak
  \textit{E-mail address:} \texttt{can.yaylali@uni-due.de}

\end{document}